\documentclass{amsart}

\usepackage{amsmath,amsthm,amssymb}
\usepackage{graphicx,color}
\usepackage{hyperref}
\usepackage{mathrsfs}
\usepackage[utf8]{inputenc}
\definecolor{Red}{cmyk}{0,1,1,0}

\definecolor{Blue}{cmyk}{1,1,0,0}

\theoremstyle{plain}
\newtheorem{theorem}{Theorem}[section]
\newtheorem{corollary}[theorem]{Corollary}
\newtheorem{proposition}[theorem]{Proposition}
\newtheorem{lemma}[theorem]{Lemma}
\newtheorem*{theorem-main}{Main Theorem}

\theoremstyle{definition}
\newtheorem{definition}[theorem]{Definition}
\newtheorem{remark}[theorem]{Remark}
\newtheorem{example}[theorem]{Example}

\newtheorem*{condicao*}{Decay condition}

\title[General Topological Markov chains]
{Some Probabilistic Properties of General Topological Markov Chains }
\author[E. A. Silva]{E. A. Silva}
\address{Departamento de Matemática, Universidade de Bras\'ilia, 70910-900, Bras\'ilia, Brazil}
\email{e.a.silva@mat.unb.br} 

\author[E. G. Mesquita]{E. G. Mesquita}
\address{Departamento de Matemática, Universidade Federal do Tocantins, 77330-000, \break \phantom{Ari}\!Arraias, Brazil}
\email{elisgardel@uft.edu.br} 

\author[E. Matias]{E. Matias}
\address{Departamento de Matemática, Universidade Federal da Bahia, 70910-900, Salvador, Brazil}
\email{edgar.matias@ufba.br} 

\date{}

\begin{document}

\maketitle
	

\begin{abstract}
 In this paper, we utilize the framework of Markov processes to attain a more probabilistic perspective on the theory of transfer operators. In doing so, we establish a functional central limit theorem (FLCT) for an $O(N)$ model associated with Dyson potential on the one-dimensional lattice. We also proof  a FLCT  on   a non-compact alphabet setting,  for a model associated with a Dyson type potential on the one-dimensional lattice. A Breimann ergodic theorem for equilibrium measures arising from the Ruelle Operator Formalism is proved. Furthermore, we obtain a qualitative criterion (strong transitivity) to determine when the conformal measure has full support. As an application, we show how to connect the Ruelle operator framework with the perspective of the {\it Hopf theory of Markov Processes.}
\end{abstract}

\section{introduction}

The classical Thermodynamic Formalism has its start in the seminal paper of David Ruelle \cite{MR0234697}
on the one-dimensional lattice gas model  and further developed for subshifts of finite type, i.e, some suitable subsets of  $\Sigma=\{1, \ldots,k\}$.

Since then, the theory has been widely developed in several ways, particularly in considering alphabets beyond the finite ones.  Restricting ourselves to works that consider  countable alphabets, we can cite for  instance, \cite{MR1818392,Sarig:2013, Sarig:2001a,MR1853808,  MR1955261}, where a  Thermodynamic Formalism theory is developed for countable alphabets. 

From the viewpoint of Statistical Mechanics it is also of interest  to consider shift spaces being more general spaces, such as  products of compact metric alphabets or even  standard Borel spaces.  A Thermodynamic Formalism for these types of alphabets has been extensively developed, we cite  for example  \cite{MR2864625,MR3377291,MR3538412,CSS19,MR4246976,MR4344719}. In \cite{MR2864625} a Ruelle operator formalism was developed for the alphabet $E=\mathbb{S}^1$, extending the classical theory to the $XY$ model and in a subsequent work \cite{MR3377291}   extended for general compact metric alphabets. In the work  \cite{MR4246976} the transfer operator framework is used to describe an  $O(N)$ model (see \cite[Chapter 9]{MR2807681}).   The relation between the Gibbs measures on the one-dimensional lattice  obtained through Ruelle operators and Gibbs measures in the Dobrushin-Lanford-Ruelle (DLR) sense is established  in \cite{MR4112024} for compact alphabets (in the lattice $\mathbb{Z}$ this equivalence is false, see \cite{Rodrigo}).

More recently in \cite{CSS19} a thermodynamic formalism for alphabets being  standard Borel spaces (non necessarily compact) is  developed. Additionally, it is worth mentioning the work \cite{MR4344719} where  a highly non trivial application of  transfer operators framework  in general state space  is performed to obtain  fully supported ergodic  measures for a class of Linear Dynamics in Banach spaces(whose  fibers are not compact).

%

Roughly speaking, one of the principal aspects behind the transfer operator method is based on utilizing functional analysis techniques (such as Kato's perturbation theory, Tychonov-Schauder theorem, Brower fixed point theorem, Arzelà-Ascoli theorem, etc.) to establish statistical properties of Markov processes that arise from the theories of dynamical systems and statistical mechanics, in fact see \cite{MR0234697, HennionHerve:2001, hopf, Nagaev} for classical references and \cite{Baladi, MR2478676, Kloeffec, CalcuFT, Barrientos2022-cr} for more recently ones.
Our  goal in this paper is to take the opposite approach, using the framework of Markov processes to gain a more probabilistic perspective on the theory of transfer operators. In doing so, we aim to establish a more natural connection with the literature of Stochastic Processes.

 In this work, we will continue to focus on general alphabets, particularly the non-compact ones.
Motivated by the framework introduced by  \cite{kloeckner_2020} we will express the Ruelle  operator as an integral against a probability kernel.
A bit more precisely, if $T:X\to X$ measurable map  and $\mathsf{M}=(m_x)_{x\in X}$  is   the transition probability kernel  arising  from a walk on the pre-images $T^{-1}(x)$, given a  bounded continuous potential $f$ we define the \emph{transfer
	operator} of $\mathsf{M}$ with respect to $f$ by
\begin{equation}
\mathscr{L}_{\mathsf{M},f}\varphi(x)=\int_Xe^{f(y)}\varphi(y)\; dm_x(y).
\end{equation}

We find this description of transfer operators (especially those originating from dynamical systems theory) to be  very auspicious and insightful.

\subsection*{A brief description of the paper organization and main results.}

\subsubsection*{Section 2}
In this section, we present the fundamental concepts and framework of the paper.

\subsubsection*{Section 3}
In Section 3  the notion of {\it irreducible Feller process} is introduced. We show that every {\it conformal measure} (see \cite{MR1014246})  arising from a  irreducible Feller process is fully supported. By relating this notion of irreducibility  with a strong form of transitivity  we show  that every conformal measure $\nu$ of a {\it hypercyclic backward  weighted shift} is fully supported. As an outcome we are able to  extend the Ruelle operator  for  the backward weighted shift (introduced in \cite{MR4344719}) to $L^p(\nu)$. This is a new result and it opens up possibilities for the study of Linear Dynamics  from the point of view  of the {\it Hopf theory of Markov Processes}, see \cite{hopf,foguel} and  \cite[Section 2]{MR4246976}.


\subsubsection*{Section 4}
In this section, as an immediate consequence of viewing the transfer operator as a transition probability kernel we get, using a Furstenbeg's result  (see \cite{MR163345}),  the validity of a Breimann Ergodic theorem. By carefully analysis of  Furstenberg's proof, we see that his result works even in presence  of more than one stationary measures. 

\subsubsection*{Section 5}
This section is devoted to proof a Functional Central Limit theorem  for  some long-range statistical mechanics   models
 on the one-dimensional lattice. 
 
 Most commonly in the transfer operator scenario people when want to prove a CLT falls back on Nagaev-Guivarc'h method. Nevertheless   that method pass by to find a suitable subspace where the transfer operator restricted to it will be a  {\it quasi-compact}, which is in general the difficult part.

 Here we will follow the  method presented in \cite{BhLe}, where  presence of  {\it spectral gap} is not required. It  consists in to find  a solution of  the {\it Poisson equation}, 
 \[
(I-P)\upsilon=\phi.
\]
Above $P$ represents the extension of the Ruelle operator to $L^2(\mu)$, and $\mu$ is the equilibrium measure (fixed point of the transfer operator) of the model in consideration.

 In Section \ref{O(N)-model} we  consider the $O(N)$ model, $N\geqslant2,$ whose the fibers are the compact space $E=\mathbb{S}^{N-1}$, associated to a long-range {\it Dyson type potential},
\[
f(x) = \sum_{n=1}^{\infty} \dfrac{\langle x_1,x_{n+1}\rangle}{n^{2+\varepsilon}},
\] 
where $\langle x_1,x_{n+1}\rangle$ means the inner product of $\mathbb{R}^n$. The approach that we use is close to the presented in  \cite[Section 4]{MR4246976}.  By using Theorem 4.1 of \cite{kloeckner_2020} we show, in a suitable space (potentials with logarithmic modulus of continuity),  a polynomial decay of correlation for this model. We use that feature to construct a solution of the Poisson equation and ensure a validity of a FCLT to this model. 

To achieve this task we need to extend some notions and results of \cite[Section 5.1]{kloeckner_2020}, to contemplate uncountable alphabets, see Section \ref{Decay-Rate-Uncout}.

In Section \ref{CLT-NON-COMP} we extend (under some  boundedness condition) the main result of  \cite{kloeckner_2020} to the { non-compact setting.} This constitutes a non-trivial extension of Theorem 3.1 of \cite{CSS19}. Utilizing the aforementioned generalization in the non-compact setting, we employ the same strategy mentioned earlier to establish a Functional Central Limit Theorem (FCLT) for an adapted Dyson-type model.

\section{Preliminaries}\label{sec-preliminares}

A measurable space $(X,\mathscr{X})$ is a standard  Borel space 
if there exists a metric $d_X$ such that $(X,d_X)$ is a complete
separable metric space and $\mathscr{X}$ is the 
Borel sigma-algebra induced by $d_{X}$. 
An important example is the product space $X=E^\mathbb{N}$ 
 regarded as a
metric space with the  metric 
\[
d_{X}(x,y)
=
\sum_{n=1}^{\infty} \frac{1}{2^n}\min\{d_{E}(x_n,y_n),1\},
\]
where $E$ can be taken as either  
a finite set $\{1,\ldots, d\}$, the set of positive integers $\mathbb{N}$, a 
compact metric space on the  Euclidean space $\mathbb{R}^d$, or more generally a standard Borel space.
Other important example are 
$X=\ell^p({ \mathbb{N}})=\{(x_n)_{n\in \mathbb{N}}\in \mathbb{R}^{\mathbb{N}}: \sum_{n=1}^\infty |x_n|^p<\infty\}$, $1\leqslant p<\infty$  
equipped with the norm 
\[
\|(x_n)_{n\in \mathbb{N}}\|_p=\left(\sum_{n=1}^\infty |x_n|^p\right)^{1/p}.
\]

From now on $X$ denotes a standard Borel space and 
$T:X\to X$ a measurable transformation.
We denote by $\mathscr{P}(X)$ the set of probability measures on the Borel sigma-algebra $\mathscr{B}(X)$  endowed
with the weak-$*$-topology. Given $\mu\in \mathscr{P}(X)$ and a bounded Borel-measurable function
$f:X\to \mathbb{R}$, we denote by  either
$\int_X f \,d\mu$ or $\mu(f)$  the Lebesgue integral of $f$, with respect to $\mu$. 

We denote by $\mathcal{L}^p(\mu)$, $1\leqslant p<\infty$,  the space of all measurable functions such that, $\int |f|^p\, d\mu<\infty$, and by $L(\mu)=\mathcal{L}^p(\mu)/\sim$, the quotient space where $\sim$ is the equivalence relation 
defined by $f\sim g$ iff $\int|f|^p\, d\mu= \int|g|\, d\mu$, endowed  with its standard norms.
   
\subsection{Generalized H\"older spaces}

Recall that a \emph{modulus of continuity} is a continuous, increasing, concave function $\omega:[0, \infty)\to [0, \infty)$
such that $\omega(0)=0$.  We say that a function $f:X\to \mathbb{R}$ is $\omega$-H\"older continuous function 
if for some positive constant $C$, we have
\[
|f(x)-f(y)|\leqslant C\omega(d(x,y)) \quad \text{for all } x,y\in X.
\]
If $f$ is a $\omega$-H\"older continuous function,  we define 
\[ \text{Hol}_\omega(f)=\sup_{x\neq y}\frac{|f(x)-f(y)|}{\omega(d(x,y))}.\] 
We refer to $C_{b}(X,\mathbb{R})$ as the Banach space of all 
real-valued bounded continuous functions endowed with its  supremum
norm $\|f\|_\infty=\sup_{x\in X}|f(x)|$.
We use the notation $C^{\omega}(X,\mathbb{R})$
to denote the space of all bounded real-valued $\omega$-H\"older functions on $X$, endowed with the norm
\[
\|f\|_\omega=\|f\|_\infty+\text{Hol}_\omega(f).
\]
It is easy to see that the normed  space $(C^\omega(X,\mathbb{R}),\|\cdot\|_{\omega})$ is a Banach algebra, that is, 
it is a Banach space and the following inequality holds $\|fg\|_\omega\leq \|f\|_\omega\|g\|_\omega$, 
for all $f,g \in C^\omega(X,\mathbb{R})$.

Two particularly important examples of moduli of continuity are for, $0<\alpha\leqslant 1$ and $\varepsilon>0$,  $\omega(t)=t^{\alpha},$ and, $\omega_{\varepsilon}(t)= \log(t_0/t)^{-\varepsilon}$ for $t_0$ sufficiently large and $t\in [0,1]$. We denote by $C^{\varepsilon \log}(X, \mathbb{R})$ the space of all bounded continuous potentials having $\omega_{\varepsilon}$ as modulus of continuity.

\subsection{ Markov chains on general state space framework }

A transition probability kernel on a standard Borel space $(X, \mathscr{X})$ means a family $\mathsf{M}=(m_x)_{x\in X}$ of probability measures such that for each $A\in\mathscr{B}(X)$ the function $x\mapsto m_x( A)\equiv m(x,A)$ is a measurable function.

Let $\mu \in \mathscr{P}(X)$ be an initial probability and   $\mathsf{M}=(m_x)_{x\in X}$ a transition probability kernel. A time homogeneous stochastic process $\Phi\equiv \{\Phi_0,\Phi_1,\ldots\}$ on a probability  space $(X,\mathscr{X},\mathbb{P}_{\mu})$,  whose  its finite dimensional distributions satisfy,
for each $n\geq 1$,
\begin{multline}\label{Markov-Chain}
\mathbb{P}_{\mu}(\Phi_0\in A_{0},\ldots,\Phi_n\in A_{n})
\\=
\int_{A_0}\ldots\int_{A_{n-1}}
m(y_{n-1},A_n)dm(y_{n-2},y_{n-1})\ldots dm(y_0,y_1)d\mu(y_0), 
\end{multline} 
is called a {\it Markov chain} with initial distribution $\mu$ and transition kernel $\mathsf{M}=(m_x)_{x\in X}.$  
If $\mu=\delta_x$ for some $x\in X$ we say that $\Phi$ is a Markov Chain {\it starting at} $x$.

Given an initial probability measure  $\mu \in \mathscr{P}(X)$ and   $\mathsf{M}=(m_x)_{x\in X}$ a transition probability kernel we can always construct  on $(X^{\mathbb{N}},\mathscr{B}(X^{\mathbb{N}}))$  a probability measure $\mathbb{P}_{\mu}$ and a stochastic process $\Phi$ whose its finite dimensional distributions satisfy (\ref{Markov-Chain}), see \cite[Theorem 3.4.1]{MR2509253}.

The transition probability  kernel $\mathsf{M}=(m_x)_{x\in X}$ is said to have the {\it Feller property}   when induces for every $f\in C_b(X,\mathbb{R})$ a bounded continuous map defined by 
\[
x\mapsto Pf(x)=\int f(y)m_x(dy).
\]
 It is clearly that  the correspondence $f\mapsto Pf$ induces a map $P:C_b(X,\mathbb{R})\to C_b(X,\mathbb{R})$, called the {\it Feller operator} associated to $\mathsf{M}=(m_x)_{x\in X}.$

The {\it Markov operator} induced by $m(x, A)$ is the map $\mathscr{T}:\mathscr{P}(X)\to \mathscr{P}(X)$ defined  by, 
\[
\mathscr{T}\mu(A)=\int_X m_x(A)\, d\mu(x).
\]
It is a straightforward calculation to show that $\mathscr{T}$ must satisfy the duality relation
\[
\int_X f(x)\,d \mathscr{T}\mu(x)=\int_X Pf(x)d\mu(x),
\]
for every $f\in C_b(X)$.
A  probability measure $\mu \in \mathscr{P}(X)$  satisfying, $\mathscr{T}\mu=\mu$, i.e, 
\[
\mathscr{T}\mu(A)=\int_X m_x(A)d\mu(x)=\mu(A), \quad \forall A\in \mathscr{X}, 
\]
is called a {\it stationary measure.}

\subsection{General transfer operator and its basic properties}\label{introduzindo-operador}

As we already mention, in this paper we use the framework of \cite{kloeckner_2020}, which we describe below.

A transition probability  kernel $\mathsf{M}=(m_x)_{x\in X}$ is said to be a \emph{backward walk} of a measurable map $T:X\to X$ if for all $x\in X$ the measure $m_x$ is fully supported on 
$T^{-1}(x)$. Given a potential $f\in C_b(X,\mathbb{R})$ we define the \emph{transfer
operator} of $\mathsf{M}$ with respect to $f$ by
\begin{equation}
\mathscr{L}_{\mathsf{M},f}\varphi(x)=\int_Xe^{f(y)}\varphi(y)\; dm_x(y).
\end{equation}

For each $n\in\mathbb{N}$ and $x\in X$ we denote by $m_x^n$ the measure on $X^n$ which 
is the law of a Markov chain $\Phi=(\Phi_n)$ starting at $\Phi_0=x$ 
and following the transition probability  kernel $\mathsf{M}$. 
In other words, denoting by $\bar x=(x_1,\dots,x_n)$ points of $X^n$, $m_x^n$ is defined by
its action on bounded continuous functions $F:X^n\to\mathbb{R}$ as follows
\[
\int_{X^n} F(\bar x)\,  d m_x^n(\bar x) 
= \int_{X}\cdots\int_{X} F(\bar x)\, d m_{x_{n-1}}(x_n) d m_{x_{n-2}}(x_{n-1}) \cdots  dm_x(x_1).
\]
For the sake of simplicity, we will denote 
\[
(m_x^n)_{x\in X} = \underbrace{(m_x)_{x\in X} \circ (m_x)_{x\in X} \circ \dots \circ (m_x)_{x\in X}}_{n\mbox{ - times}} \quad\text{or}\quad \mathsf{M}^n = \mathsf{M} \circ\dots\circ \mathsf{M}.
\]

 A straightforward computation shows that $\mathscr{L}_{\mathsf{M},f}^n$ can be expressed in terms of Birkhoff sums as, 
\[\mathscr{L}_{\mathsf{M},f}^n \varphi(x) = \int_{X^n} e^{f^n(\bar x)} \varphi(x_n)\,d m^n_x(\bar x)\]
where
$f^n: X^n \to \mathbb{R}$ is the map defined by 
 $\bar x=(x_1,\dots,x_n)\longmapsto f(x_1)+\dots + f(x_n)$.

Denote by $\varrho$ the spectral radius of the transfer operator $\mathscr{L}_{\mathsf{M},f}.$  Following Denker and Urbanski \cite{MR1014246}, we call a  
probability  measure $\nu \in \mathscr{P}(X)$ belonging to the set 
\[
\mathscr{G}^*(\varrho)=\{\nu \in \mathscr{P}(X) : \mathscr{L}_{\mathsf{M},f}^*\nu=\varrho \nu\}
\]
   a {\it $\varrho$-conformal measure}, or simply a conformal measure. It is a well known result that when $X$ is a compact and $T$ is continuous the set $\mathscr{G}^*$ is  nonempty. We will see in next section two examples where  $X$ is non-compact and $\mathscr{G}^*$ still nonempty.


\subsection{A  natural Markov chain}

We will be particularly interested in potentials that satisfy $\mathscr{L}_{\mathsf{M},f}\mathbf{1}=\mathbf{1}$;  these potentials are called {\it normalized}. This notion can be stated in a different manner, by saying that $e^{f}dm_x$  is a probability measure for every $x$.

Let $\mathsf{M}=(m_x)_{x\in X}$ be a backward walk for a dynamical system $T$.  When $f$ is normalized an other natural stochastic process associated to the transfer operator $\mathscr{L}_{\mathsf{M},f}$ is the Markov chain  $(\Psi_n)$ having $\widetilde{m}_x(A)=\mathscr{L}_{\mathsf{M},f}(\mathbf{1}_A)(x)$, $A\in \mathscr{X}$ and $x\in X$, as a transition kernel.

Below,  the expectation with respect to the joint law of the  Markov
Chain $(\Psi_n)_{n\in\mathbb{N}}$ with stationary measure
$\mu$ will be denoted as usual by $\mathbb{E}_{\mu}$.
The distributional relation between the Markov chain $(\Psi_n)_{n\in\mathbb{N}}$
and the underlying dynamics is given
by the following lemma, whose the proof can be found in reference \cite[p.85]{HennionHerve:2001}.
\begin{lemma}\label{lema-eq-distribution}
	Let $(\Psi_n)_{n\in\mathbb{N}}$ be the Markov chain with transition kernel  $\widetilde{\mathsf{M}}=(\widetilde{m}_x)_{x\in X}$ defined above,
	$n\geqslant 1$, $g:X^{n}\to \mathbb{R}$ a positive measurable function and $\mu$ the stationary measure.
	Then we have
	\[
	\int_X g(x,T(x), \ldots, T^{n-1}(x))\,d\mu(x)=\mathbb{E}_\mu[g(\Psi_n, \Psi_{n-1},\ldots,\Psi_1 )].
	\]
\end{lemma}

Here we will focus on functions $g$ of the form $g ={\mathbf{1}}_{A}\circ h$, where
$A=(-\infty,t]$ is some suitable interval on the real line, $h:X^n\to \mathbb{R}$ is given by
$h(z_1,\ldots, z_n) = \phi(z_1)+\ldots +\phi(z_n)$, with $\phi:X\to\mathbb{R}$ being a positive
function in some Banach space, for example, the space of  H\"older continuous functions.

In these particular cases, by taking $A=(-\infty,t\sqrt{n}]$, and $h$ as above, we get from
Lemma \ref{lema-eq-distribution} the following identity
\begin{align*}
\mu(\{\textstyle\sum_{j=0}^{n-1}\phi\circ T^{j}(x)\leqslant t\sqrt{n} \})
& =
\int_X {\mathbf{1}}_{(-\infty,t\sqrt{n}]}\circ h(x,T(x), \ldots, T^{n-1}(x))\,d\mu(x)
\\[0.3cm]
& =\mathbb{E}_\mu[{\mathbf{1}}_{(-\infty,t\sqrt{n}]}\circ h(\Psi_n, \Psi_{n-1},\ldots,\Psi_1 )]
\\[0.3cm]
& =
\mathbb{P}_{\mu}(\{ \textstyle\sum_{j=1}^n\phi(\Psi_j)\leqslant t\sqrt{n}\}).
\end{align*}

This relation implies, for example, that if random variables
$(\phi(\Psi_n))_{n\in\mathbb{N}}$ are distributed according to
$\mathbb{P}_{\mu}$ and obeys a CLT then
for  the random variables  $(\phi\circ T^n)_{n\in\mathbb{N}}$
distributed according to $\mu$ also obeys a CLT. 

\section{ Irreducible Feller processes.}


Understanding the conformal measures of a dynamical system allows for a deeper comprehension of its geometric, dynamical, and statistical characteristics. It provides insights into the long-term behavior, and fundamental structures of the system.
One important feature of conformal measures or stationary measures is their support. Next we provide  a general probabilistic condition sufficient   to characterize the support of stationary and  conformal measures of general transfer operators.
\subsection{Irreducible Feller processes }\label{seccao-suporte}
For the sake of convenience it will be useful to   extend the notion of conformal measure to bit more general setting.
Let $\mathsf{M}=(m_x)_{x\in X}$ be a transition probability kernel on a Borel standard space $X$,   a probability measure $\nu \in \mathscr{P}(X)$  is called a conformal measure
to the Markov operator $\mathscr{T}$ if there exists  $\lambda>0$ such that, $\mathscr{T}\nu =\lambda\nu$.

\begin{definition}
	We say that a Markov chain $\Phi=(\Phi_n)$ with  transition kernel $\mathsf{M}=(m_x)_{x\in X}$ satisfying  the  Feller property  is {\it irreducible} when for every non-negative and nontrivial bounded continuous function $\varphi$, and  for every $x\in X$ there exists $n\in \mathbb{N}$ such that $P^n\varphi(x)>0$.
	\end{definition}



\begin{theorem}\label{Crit-Supp-Tot}
	Let $\Phi=(\Phi_n)$ be a  Markov chain  having  $\mathsf{M}=(m_x)_{x\in X}$ as transition kernel, and satisfying the Feller Property.  Suppose that $\Phi=(\Phi_n)$ irreducible and that its Feller operator $P:C_b(X, \mathbb{R})\to C_b(X, \mathbb{R})$ is a bounded linear operator. Then every conformal measure is fully supported.
\end{theorem}

\begin{proof}
Let $\lambda> 0$ and $\nu\in \mathscr{P}(X)$ be such that $\mathscr{T}\nu=\lambda \nu$.  We can assume, by making a change of scale if necessary, that $0<\lambda<1$.	Then we have for every $n\geqslant 1$, and for every non-negative and non-trivial bounded continuous function $\varphi$ we have  that, 
\[
\lambda^n\nu(\varphi)=\mathscr{T}^n \nu(\varphi)=\int_X P^n\varphi(y) \,d\nu(y).
\]
By summing on $n$ on both sides of the above equality, we get that
\[
\dfrac{\lambda}{1-\lambda}\nu(\varphi)=\sum_{n=1}^\infty \int_X P^n\varphi(y) \,d\nu(y)=\int_X\sum_{n=1}^\infty  P^n\varphi(y)\, d\nu(y),
\]
where the last equality above is guaranteed by the monotone convergence theorem.
Now, since $P$ preserves the cone of the non-negative functions, we have that  
$G(y)=\sum_{n=1}^\infty  P^n\varphi(y)$ is  a non-negative measurable  function. Remembering that $\Phi$ is irreducible we get that $G(y)>0$ for every $y\in X$, thus implying that $\nu(\varphi)>0$. 
\end{proof}


\begin{proposition}\label{suporte-total}
	Let $f\in C_b(X, \mathbb{R})$ be  a continuous function and $\mathsf{M}=(m_x)_{x\in X}$ a backward walk of some measurable map $T:X\to X$. Consider the natural Markov chains $(\Phi_n)$ and $(\Psi_n)$ associated to the transfer operator $\mathscr{L}_{\mathsf{M},f}$ introduced in  previous section. If the process $(\Phi_n)$ is Feller irreducible then  $(\Psi_n)$ has the same property.
\end{proposition}

\begin{proof}
 Since $(\Phi_n)$ is irreducible, given $x\in X$ there exists an integer $M\equiv M(x)$ such that $\mathscr{L}_{\mathsf{M},0}^M\varphi(x)>0$. Thus, we have the following, 
\begin{multline}
\mathscr{L}_{\mathsf{M},f}^M \varphi(x)=
\int_{X^M} e^{f^M(\bar x)}\varphi(x_n)\,d m^M_x(\bar x)
\\
\geqslant
\bigl(\inf  e^{f^M(\bar x)}\bigl) \int_{X^M}  \varphi(x_n)\,d m^M_x(\bar x)\,
\\= \inf \bigl( e^{f^M(\bar x)}\bigr)\mathscr{L}_{\mathsf{M},0}^M\varphi(x)>0,
\end{multline}
which is the desired result.
\end{proof}

\subsection{The standard Ruelle operator}\label{Ruelle-padrao}
	Let $E$ be a standard Borel space, and $\mu$ an a priori probability measure  in $\mathscr{P}(E)$. For each $x\in  X=E^{\mathbb{N}}$ consider the probability measure  $m_x = \mu\times \delta_x$ on $\mathscr{B}(X)$. Then $(m_x)_{x\in X}$ is a transition probability kernel and for each $f\in C_b(X, \mathbb{R})$ fixed we have that, 
	\begin{align*}
	\mathscr{L}_{\mathsf{M}, f} \varphi(x) 
	&= \int_{X} e^{f(y)} \varphi(y) d m_x (y)
	= \int_{E^{\mathbb{N}}} e^{f(y)} \varphi(y)\, d m_x (y)\\[0.3cm]
	&=\int_{E\times E^{\mathbb{N}}} e^{f(ay)} \varphi(ay)\, d \mu(a) \, d\delta_x (y)
	=\int_{E}\left[\int_{E^{\mathbb{N}}} e^{f(ay)} \varphi(ay)\, d\delta_x (y)\right] d\mu(a) \\[0.3cm]
	&=\int_{E} e^{f(ax)} \varphi(ax)\, d \mu(a).
	\end{align*}
	Therefore the transfer operator associated to   $\mathsf{M}=(\mu\times \delta_x)_{x\in X}$ coincides with
	 the Ruelle operator for the shift map considered in \cite{MR2864625, MR3538412, MR3377291, CSS19}.  
In each one of the previous references is showed that for a sufficiently regular  potential $f$,  we have that $\mathscr{G}^*(\varrho_f)\neq \varnothing.$

		Observe that  a realization of the Markov process associated to the transition kernel $\mathsf{M}=(\mu\times \delta_x)_{x\in X}$ started at $x$ is given by $\Phi_n=a_1\cdots a_n x$. Below by  showing  that the process $(\Phi_n)$ is irreducible we  will proof that  any conformal measure $\nu\in \mathscr{G}^*(\varrho_f)$ is fully supported.
	
\begin{proposition}
Consider the following  backward   walk $\mathsf{M}=(\mu\times \delta_{x})_{x\in X}$  for the full shift and let $f\in C_b(X, \mathbb{R})$ be a bounded continuous potential. Then any conformal measure with respect to  $\mathscr{L}_{\mathsf{M},f}$ is fully supported.
\end{proposition}

\begin{proof}
		It is enough to consider a non-negative and non-trivial continuous function $\varphi \in C_b(X, \mathbb{R})$. By continuity there exists an open set $U$ such that $\varphi$ is strictly positive in $U$.
On the other side, one can easily verify  that for each $x\in X$ there always exists a positive integer $M$ such that $\sigma^{-M}x\cap U\neq \varnothing,$ i.e, there exists $(a_1, \ldots, a_M)\in \mathbb{R}^M$ so that, $\varphi(a_1a_2\ldots a_M x)>0$.  Again, by continuity we can find a small neighborhood $W$ of $({a}_1,\ldots, {a}_M)$ in $\mathbb{R}^M$ such that for every $(\widetilde{a_1},\ldots, \widetilde{a}_M)\in W$ we have that $\widetilde{a_1}\cdots \widetilde{a}_Mx$ still in $U$.
	Then we have the following, 
\begin{align*}
\mathscr{L}_{\mathsf{M},0}^{M}\varphi(x)=\int_{\mathbb{R}^M}\varphi(a_1\cdots a_Mx)\, \mu^M(a_1\cdots a_M)
\\
\geqslant
\int_{W}\varphi(a_1\cdots a_Mx)\, \mu^M(a_1\cdots a_M)>0.
\end{align*}
  Since $\mu$ is fully supported we have that the above integral is positive.
From that by using Proposition \ref{suporte-total}   and Theorem \ref{Crit-Supp-Tot} we conclude that for each continuous potential $f$,  any conformal measure $\nu\in \mathscr{G}^*(\varrho_f)$ is fully supported. 
\end{proof}

\subsection{Ruelle operator of the backward weighted shift}\label{wheighted-shift}
	In this section $X$ will denote  $\ell^p({\mathbb{N}})$ with $1 \leqslant p < \infty.$
	
	We fix values $0 < c < c'$ and  a weight sequence $(\alpha_n)_{n \geqslant 1}$ satisfying $\alpha_n \in (c, c')$ for each $n \in \mathbb{N}$. The {\it backward weighted shift} associated to the sequence $(\alpha_n)_{n \geqslant 1}$  
	is defined as the linear map $L : X \to X$ given by
	\[
	L((x_n)_{n \geqslant 1}) = (\alpha_n x_{n+1})_{n \geqslant 1} \,.
	\]

	For a weight sequence $(\alpha_n)_{n\geqslant 1}$ 
	 we consider the auxiliary  map $S:X\to X$ defined as 
	$
	S(x)=(\alpha_1^{-1}x_1,\alpha_2^{-1}x_2, \ldots).
	$
		Let  $\mu$ be  a fully supported  Borel  a priori measure in $\mathbb{R}$  and let  $\mathsf{M}=(m_x)_{x\in X}$ be the  transition probability kernel defined as follows, 
	\[
	m_x=\delta_{S(x)}\times \mu.
	\]
Let $f\in C_b(X, \mathbb{R})$ be a bounded continuous potential, by observing that $\mathbb{R}\times \ell^p({ \mathbb{N}})$ is isometric to  $\ell^p({ \mathbb{N}})$ 
	we have that,
	\begin{multline*}
	\mathscr{L}_{\mathsf{M}, f}\varphi(x) 
	= \int_{\ell^p({\mathbb{N}})} e^{f(y)} \varphi(y)\, d m_x (y)
	=
	\int_{\mathbb{R}\times \ell^p({ \mathbb{N}})} e^{f(ry)} \varphi(ry)\, d \mu(r) \, d\delta_{S(x)} (y)
	\\[0.3cm]
	=\int_{\mathbb{R}}\left[\int_{\ell^p({\mathbb{N}})} e^{f(ry)} \varphi(ry)\, d\delta_{S(x)} (y)\right]\, d\mu(r)
	=\int_{\mathbb{R}} e^{f(rS(x))} \varphi(rS(x))\, d \mu(r)
	\\[0.3cm]
	=
	\int_{\mathbb{R}} e^{f\bigl(r, \frac{x_1}{\alpha_1}, \frac{x_2}{\alpha_2}, \ldots\bigr)}\varphi\bigl(r, \frac{x_1}{\alpha_1}, \frac{x_2}{\alpha_2}, \ldots\bigr) d \mu(r).
	\end{multline*} 
	which coincides with the Ruelle operator associated to the weighted shift  introduced in \cite{MR4344719}.
	In this case a realization of the Markov process associated to the transition kernel $\mathsf{M}=(\mu\times \delta_{S(x)})_{x\in X}$ started at $x$ is given by
	$\Phi_n=\bigl(r_1, \ldots, r_n, \frac{x_1}{\beta_1^n}, \frac{x_2}{\beta_2^n}, \ldots\bigr)$, where $(r_1, \ldots, r_n)\in \mathbb{R}^n$. 
	
		Suppose $0 < c < c'$ and $1 < c'$. 
		For each $k,n\in\mathbb{N}$ we define  
		$\beta_k^n \equiv  \alpha_k \ldots \alpha_{k+n-1}$ and 
		\begin{eqnarray}\label{dnn}
		\displaystyle d_n = \inf_{k \geqslant 1} \beta_k^n.
		\end{eqnarray}
			In \cite{MR4344719} the authors consider a  weight  sequence $(\alpha_n)_{n \geqslant 1}$ that satisfy $\alpha_n \in (c, c')$, for all $n \in \mathbb{N}$ and
	\begin{eqnarray}
	\label{d}
 \displaystyle \sum_{n =1}^{\infty} (d_n)^{-\alpha} < \infty.
	\end{eqnarray}
	They showed, under suitable conditions on the a priori measure $\mu$, and for  a locally Hölder potential $f$,  that the Ruelle operator   $\mathscr{L}_{\mathsf{M}, f}$ has at least one equilibrium measure (ergodic).
	
	It is worth mentioning that condition (\ref{d}) implies that $L$ is {\it frequently hypercyclic}\footnote{This is a standard terminology from  Dynamics of Linear Operators Theory,  see \cite{LinDy}}  by meaning that, 
	there exists $x\in X$ such that for any open set $V$ the set $N(x, V)=\{n: L^nx\in V\}$ has positive lower density, i.e
\[
	\liminf_{n\to \infty} \dfrac{1}{n}\# N(x,V)\cap\{1,\ldots,n \}>0,
\]
 see \cite[Chapter 6]{LinDy}. The property of to be frequently hypercyclic implies in turn  that $L$ is  {\it strongly  transitive}, i.e, for any non-empty  open set $U$ we have that, 	
\begin{equation}\label{Strong-Transitivity}
X\setminus\{0\}\subset \bigcup_{n=1}^\infty L^n(U),
\end{equation}
 indeed see \cite[Theorem 4.4]{MR4200794}. It is useful to observe that  inclusion (\ref{Strong-Transitivity}) can be rewritten as follows, 
 \begin{equation}\label{Strong-Transitivity-2}
 \forall x\in X\setminus \{0\}, \forall U\neq \varnothing~~\text{open set}~ \exists n\in \mathbb{N},~\text{such that}~L^{-n}(x)\cap U\neq \varnothing.
 \end{equation}

We will show that $(\Phi_n)$ satisfies the conditions of  Theorem \ref{Crit-Supp-Tot}. 

\begin{theorem}
Consider the following  backward   walk $\mathsf{M}=(\mu\times \delta_{S(x)})_{x\in X}$ for a strongly transitive  weighted shift and let $f\in C_b(X, \mathbb{R})$ be a bounded continuous potential. Then any conformal measure with respect to $\mathscr{L}_{\mathsf{M},f}$ is fully supported.
\end{theorem}

\begin{proof}
Let $\varphi \in C_b(X, \mathbb{R})$ be non-negative and nontrivial. Continuity  implies the existence of an open set $U$ where $\varphi$ is strictly positive. By (\ref{Strong-Transitivity-2}) there  exists $(r_1, \ldots,r_n)\in \mathbb{R}^n$ such that   $\varphi\bigl(r_1,\ldots ,r_n, \frac{x_1}{\beta_1^n}, \frac{x_2}{\beta_2^n}, \ldots\bigr)>0$. Continuity of $\varphi$ implies that the previous inequality remains valid in some open neighborhood $V$of $(r_1, \ldots,r_n)$ in $\mathbb{R}^n$. Therefore, letting $\mu^n$ denote the product measure $\mu\times \cdots \times\mu$ we have that
	\begin{multline*}
 \mathscr{L}_{\mathsf{M},0}^{n}\varphi(x)
=\int_X \varphi(y)\, dm_x(y) =\int_{\mathbb{R}^n} \varphi  \bigl(r_1,\ldots ,r_n, \frac{x_1}{\beta_1^n}, \frac{x_2}{\beta_2^n}, \ldots\bigr) \, d\mu^n(r_1,\ldots ,r_n)                             
\\
\geqslant 
\int_{V} \varphi  \bigl(r_1,\ldots ,r_n, \frac{x_1}{\beta_1^n}, \frac{x_2}{\beta_2^n}, \ldots\bigr) \, d\mu^n(r_1,\ldots ,r_n)>0
\end{multline*}
 which means that $(\Phi_n)$ is irreducible. Then, it follows from Proposition \ref{suporte-total} and Theorem \ref{Crit-Supp-Tot} that any conformal measure of  $\mathscr{L}_{\mathsf{M},f}^{n}$  is fully supported.
\end{proof}

\begin{remark}
	A sufficient condition  for strong transitivity in the context of linear dynamics is to require $\bigcup_{n=1}^\infty \text{Ker}(T^n)$ to be dense in $X$. (see \cite[Corollary 2.2]{MR4200794}).

\end{remark}

%
%
%

\subsection{Nonexpansive IFS}\label{Nonexpansive}
In this section  $X$ denotes a compact subset of $\mathbb{R}^d$. For each $i=1, \ldots m,$ consider $T_i:X\to X$ a continuous map. Let $p_1, p_2, \ldots, p_m,$ be
 continuous positive   real functions satisfying  $\sum_{i=1}^n p_i(x)=1, \, \forall x\in X$. 
	We say that $T:X \to X$  is {\it nonexpansive} if $|T(x)-T(y)|\leqslant |x-y|$ $\forall x, y \in X$ and  {\it weakly contractive} if, 
	\[
	\alpha_T(t):=\sup_{|x-y|\leq t}|T(x)-T(y)|<t, \forall t>0.
	\]
For a multi-index $J=(j_1, \ldots, j_n),$ $1\leqslant j_k\leqslant m$, set, $T_J(x)=T_{j_1}\circ \cdots \circ T_{j_n}(x)$.
The following result  is proved in \cite[Proposition 2.1.]{MR1881258}	

\begin{proposition}\label{invariant-attractor}
Suppose that $\{T_i\}_{i=1}^m$ are continuous on $X$ and at least one of them is weakly contractive. 
	Then there exists a unique smallest nonempty compact set $K$ such that, 
	\[
	K=\bigcup_{i=1}^m T_i(K).
	\] 
	Moreover, for any $x\in K$, the closure of $\{ T_{J}(x): |J|=n, n\in \mathbb{N}\}$ is $K$.
\end{proposition}
	
	Now, for each $x\in X$ define   $m_x = \sum_{i=1}^{n}p_i(x) \delta_{T_i(x)}$ on $\mathscr{P}(K)$. Is clear that  $\mathsf{M}=(m_x)_{x\in X}$ is a transition probability kernel and we have for any $\varphi\in C(K, \mathbb{R})$ that, 
	\begin{align*}
	\mathscr{L}_{\mathsf{M}, 0} \varphi(x) 
	&= \int_{X} e^{f(y)} \varphi(y) \,d m_x (y)
	= \sum_{i=1}^{n} p_i(x) \varphi(T_i(x)),
	\end{align*}
	which turns out to be the classical transfer operator associate to the IFS generated by $\{T_1, T_2, \ldots, T_n\}.$ 
More generally if $J=(j_1, \ldots, j_n)$, $1\leqslant j_k\leqslant m$, is a multi-index and we denote 
\[
p_{T_J}(x):=p_{j_1}(T_{j_1}\circ T_{j_2}\circ \cdots \circ T_{j_n})\cdots p_{j_{n-1}}(T_{j_{n-1}} \circ T_{j_n})p_{j_n}(T_{j_n}),
\]
then we have 
\begin{align*}
\mathscr{L}_{\mathsf{M}, 0}^n \varphi(x) 
&= \sum_{|J|=n}^{n} p_{T_J}(x)\varphi(T_J(x)).
\end{align*}

We will show that $\Phi=(\Phi_n)$ is irreducible. Indeed, let $\varphi \in C(K, \mathbb{R})$ with $\varphi\geqslant 0$ and $\varphi \not \equiv 0$. Define $V=\{x\in K: \varphi(x)>0\}.$ It follows from Proposition  \ref{invariant-attractor} that there exist a multi-index $J_0$ such that $T_{J_0}(x)\in V$. Let $k=|J_0|$, then we have that
 \[
 	\mathscr{L}_{\mathsf{M}, 0}^{k} \varphi(x) =\sum_{|J|=k}p_{T_J(x)}\varphi(T^J(x))\geqslant p_{T_{J_0}}(x)\varphi(T_{J_0}(x))>0.
 \]

Therefore, as a consequence of Proposition $\ref{suporte-total}$ and Theorem $\ref{Crit-Supp-Tot}$ any conformal measure of 	$\mathscr{L}_{\mathsf{M}, 0}^{k}$ is fully supported.

 \begin{remark}
 	This remark regard Examples \ref{Ruelle-padrao}, \ref{wheighted-shift} and \ref{Nonexpansive}. One should note that  the full shift in the Example \ref{Ruelle-padrao} has a property much stronger than being  strongly transitive, indeed it holds the following,
 	\begin{equation}\label{Irreducibility-2}
 	\forall U\neq \varnothing~~\text{open set}~ \exists M=M(U)\in \mathbb{N},~\text{such that}~ \forall x\in X,\sigma^{-1}(x)\cap U\neq \varnothing, 
 	\end{equation}
 	or, in other words, for each open set $U$ there exists an integer $M$ such that we have $ \sigma^MU=X$.  	
 	This is a kind of a weak topological form of {\it Doebling's minorization condition} see \cite[Theorem 9.1]{Bhbook}. Condition (\ref{Irreducibility-2}) seems also to be related to the Bhattacharya-Waymire splitting condition, see \cite[Section 3.5.2 ]{Bhbook}. However (\ref{Irreducibility-2}) does not imply either of these two probabilistic properties because both of them implies exponential asymptotically stability of the Markov operator.

 	Obviously  property (\ref{Irreducibility-2}) is not valid for the  backward weighted shift (in fact $\text{diam}(L^kU)\leqslant \|L\|^k \text{diam}(U)$). Nevertheless $L$ has the property
 	$
 	X\setminus\{0\}\subset \bigcup_{n=1}^\infty L^n(U),
 	$
 	which is in turn a much weaker property than (\ref{Irreducibility-2}), but still a stronger form of transitivity.
 	
 	The conclusion of Proposition \ref{invariant-attractor} is also a type of strong transitivity formulated for IFS.
 	
 	These examples suggests that full support property of the conformal measure is related to this strong form transitivity. 
 	
 	

 \end{remark}

 \subsection{Application: Extension of the Transfer Operator on a non-compact setting}
%
%
%
%
%
%

Suppose  $\nu$ is  a conformal measure of some transfer operator $\mathscr{L}_{\mathsf{M},f}$. Extending  $\mathscr{L}_{\mathsf{M},f}$ to some $\mathbb{L}:L^1(\nu)\to L^1(\nu)$, having some control on the spectrum of   $\mathbb{L}$ has a very good karma. Indeed,  in doing so we can   see the $\mathbb{L}$ as a Markov process in the {\it Hopf's sense} (see \cite{hopf}).

\begin{definition}[Markov Processes]
	A Markov process is defined as an ordered quadruple $(X,\mathscr{F},\mu,T)$, 
	where the triple $(X,\mathscr{F},\mu)$ is a sigma-finite measure space with a positive measure
	$\mu$ and $T$ is a bounded linear operator acting on $L^1(\mu)$
	satisfying:
	\begin{itemize}
		\item[(i)] $T$ is a contraction: 
		$\sup\{\|T\varphi\|_{1}: \|\varphi\|_{1}\leqslant 1 \}\equiv \|T\|_{\mathrm{op}}\leqslant 1$;
		\item[(ii)] $T$ is a positive operator, that is, if $\varphi\geqslant 0$, then $T\varphi\geqslant 0$.
	\end{itemize}
\end{definition}

Here, the sigma-algebra $\mathscr{F}$ will be the Borel sigma-algebra $\mathscr{B}(X)$, 
$T$ is the extension $\mathbb{L}:L^1(\nu)\to L^1(\nu)$ of 
the transfer operator $\mathscr{L}_{\mathsf{M},f}$, and $\mu=\nu$. 
Condition (ii), the positivity property of $\mathbb{L}$,
is inherited from $\mathscr{L}_{\mathsf{M},f}$, and the condition (i) follows from  $\|\mathbb{L}\|_{\mathrm{op}}=\varrho(\mathscr{L}_{\mathsf{M},f})$,
and the assumption $\varrho (\mathscr{L}_{\mathsf{M},f})=1$.

We refer to \cite[Section 2]{MR4246976} to the reader interested in how to explore the Hopf-Markov  theory  in the Thermodynamic Formalism context.

\subsubsection{Extending the transfer operator}

In this section $X= E^\mathbb{N}$  or $X=\ell^p(\mathbb{N})$, $f\in C_b(X, \mathbb{R})$ will denote a bounded continuous potential and $\mathsf{M}$ will stand for $\mathsf{M}=(\mu \times \delta_x)_{x\in X}$ or $\mathsf{M}=(\mu \times \delta_{S(x)})_{x\in X}$.

Since	the transfer operator   $\mathscr{L}_{\mathsf{M}, f}$ is irreducible, Theorem \ref{Crit-Supp-Tot} shows that any conformal measure of $\mathscr{L}_{\mathsf{M}, f}$ is fully supported,	therefore  there is a linear continuous embedding 
	$q:(C_b(X, \mathbb{R}),\|\cdot\|_{\infty})\hookrightarrow (L^p(\nu),\|\cdot\|_{p})$.

\begin{remark}
 Theorem \ref{Crit-Supp-Tot} is fundamental to get  an embedding of $C_b(X, \mathbb{R})$ onto $(L^p(\nu),\|\cdot\|_{p})$. In general that embedding could not exists, fact if $\nu=\delta_x$ is the delta Dirac measure concentrated in some $x\in X$ then $\text{dim}_{\mathbb{R}}L^1(\nu)=1$ while   $\text{dim}_{\mathbb{R}}C_b(X, \mathbb{R})=\infty.$
\end{remark}


The most short path to provide the extension is to consider the map $L:\text{dom}(L)\subset \mathcal{L}^1(\nu)\to \mathcal{L}^1(\nu)$, defined by 
\begin{equation}\label{pre-estensao}
L\varphi(x)=\int_X e^{f(y)}\varphi(y)\, dm_x(y), \quad \forall x\in X,
\end{equation}
where  $\text{dom}(L)$ means  the subset of all $\varphi \in \mathcal{L}^1(\nu)$ such that the expression (\ref{pre-estensao}) makes sense for all $x\in X$. It is clear that $\text{dom}(L)$ is a subspace of $\mathcal{L}^1(\nu)$ and that contains $C_b(X, \mathbb{R})$.

 The map given by (\ref{pre-estensao})  seems to be quite auspicious  to fulfill our purposes, nevertheless some technical issues need to be clarified. The most fundamental question about it is whether this expression preserves $\nu$-equivalence classes. This question leads us to the following lemma.
 
 \begin{lemma}\label{quasi-invariant}
  Let $\mathsf{M}=(m_x)_{x\in X}$ be the transition probability kernel of subsection \ref{Ruelle-padrao} or \ref{wheighted-shift}. Let $\mu$ be the a priori measure and $\nu$ a conformal measure of $\mathscr{L}_{\mathsf{M}, f}.$ Then there is $K>0$ such that, 
 \[
 \mu \times \nu(B)\leqslant K \nu(B), \quad \forall B\in \mathscr{B}(X).
 \]
 \end{lemma}
 
 \begin{proof}
In case  $X=E^{\mathbb{N}}$ and $m_x=\mu\times \delta_x$ we have that $\mathscr{L}_{\mathsf{M}, f}$ is the standard Ruelle operator and the proof of the  lemma is the same as the proof of \cite[Proposition 2.2]{MR4246976}.

If  $X=\ell^p(\mathbb{N})$ and $m_x=\mu\times \delta_x,$  the proof is still very close to the one  above mentioned, but we need to use the following observation: let   $B$ be an open set  of $\ell^p(\mathbb{N})$, then every $x\in B$ is contained in a smaller open set of the form $U\times V$ where $U$ is an open interval of  $\mathbb{R}$ and $V$ is an open set of $\ell^p(\mathbb{N})$ and $U\times V\subset B$.

Define the class 
\[
\mathscr{R}=\{U\times V,\; U\subset \mathbb{R}\; \text{open interval, and }\, V\subset \ell^p(\mathbb{N})\, \text{open set}\}.
\]
We will show that the statement of the lemma is valid for every element of $\mathscr{R}$.
Let $B=U\times V \in \mathscr{R}.$ Choose an increasing sequence of continuous  functions (Urysohn functions)  $\phi_n:\mathbb{R}\to [0,1]$ and $\psi_n:\ell^{p}(\mathbb{N})\to [0,1]$, $n\in \mathbb{N},$ satisfying $\phi_n\uparrow 1_U$  and $\psi_n\uparrow 1_V$ pointwisely and in $L^1(\mu)$ and $L^1(\nu)$ respectively.  From construction we must have $\Theta_n(x)=\phi_n(x_1)\times \psi_n(\sigma(x)) \uparrow 1_B(x)$. Therefore, we have that  
\begin{align*}
\nu(B) &\geqslant \int_{\ell^p(\mathbb{N})}\Theta_n(x)\, d\nu(x)=\varrho^{-1}\int_{\ell^p(\mathbb{N})}\left[ \int_\mathbb{R}e^{f(ax)}\Theta_n(ax)\, d\mu(a)\right]\, d\nu(x)
\\
&\geqslant
\varrho^{-1}\int_{\ell^p(\mathbb{N})}\left[ \int_\mathbb{R}e^{f(ax)}\phi_n(a)\psi_n(x))\, d\mu(a)\right]\, d\nu(x)
\\
&\geqslant
 \varrho^{-1}e^{-\|f\|_\infty} \int_{\mathbb{R}} \phi_n(a)\,d\mu(a)\int_{\ell^p(\mathbb{N})} \psi_n(x)\, d\nu(x).
\end{align*}
Taking $n\to \infty$ we get that
\[
\mu\times \nu(B) \leqslant \varrho e^{\|f\|_\infty} \nu(B),
\]
concluding the lemma for the class $\mathscr{R}$. Now, observe  that $\mathscr{R}$ is a $\pi$-system which  generates the Borel sigma-algebra and consider the class $\mathscr{L}=\{B\in \mathscr{B}(\ell^p(\mathbb{N})): \mu \times \nu(B)\leqslant K \nu(B)\}$.
By observing that $\mathscr{L}$ is a $\lambda$-system and that $\mathscr{L}\supset \mathscr{R}$ we must have by the Dynkin $\pi-\lambda$ theorem that $\mathscr{B}(\ell^p(\mathbb{N}))\subset \mathscr{L}$.
 \end{proof}

With the above lemma in hands we have the following proposition.

\begin{proposition}
	The  map  $L:\textnormal{dom}(L)\subset \mathcal{L}^1(\nu)\to \mathcal{L}^1(\nu)$ given by (\ref{pre-estensao}) is well defined and  preserves $\nu$-equivalence classes of $L^1(\nu)$.
\end{proposition}

\begin{proof}
	Again, when $\mathscr{L}_{\mathsf{M}, f}$ is the standard Ruelle operator the proof is still the same as the one that we gave in  \cite[Section 2]{MR4246976}. If we consider  the operator $\mathscr{L}_{\mathsf{M}, f}$ as the transfer operator of the backward weighted shift then we have some minor changes. 
	
		Let $\varphi \in \text{dom}(L)$.  By Lemma \ref{quasi-invariant} we have
	\begin{align}\label{est-bem-def}
	\int_X |L\varphi|\, d\nu \leqslant \int_{\ell^p(\mathbb{N})} \int_\mathbb{R}e^{f\bigl(r, \frac{x_1}{\alpha_1}, \frac{x_2}{\alpha_2}, \ldots\bigr)}\varphi\bigl(r, \frac{x_1}{\alpha_1}, \frac{x_2}{\alpha_2}, \ldots\bigr) \,d \mu(r) d\nu(x)\nonumber
	\\
	\leqslant 
 e^{\|f\|_\infty} \int_{\ell^p(\mathbb{N})} \int_\mathbb{R} \varphi\bigl(r, \frac{x_1}{\alpha_1}, \frac{x_2}{\alpha_2}, \ldots\bigr)  \,d\mu(r) d\nu(x) \leqslant K e^{\|f\|_\infty} \int_{\ell^p(\mathbb{N})}  |\varphi| \,d\nu.
	\end{align}
	
The inequality (\ref{est-bem-def}) has two important consequences. The first is that $L$ is well defined. Secondly, if $\varphi=\psi$ $\nu$-a.e are $\nu$-integrable functions then $L\varphi=L\psi$ $\nu$-a.e thus implying that $L$ preserves $\nu$-equivalence classes. 
\end{proof}

Since the space of all Lipschitz functions is dense in $L^1(\nu)$, we have that $L$ induces an operator $\widetilde{\mathbb{L}}$ on a dense subset of $L^1(\nu)$. By using that $\widetilde{\mathbb{L}}$ is Lispchitz one can extend $\widetilde{\mathbb{L}}$ to a bounded operator $\mathbb{L}:L^1(\nu)\to L^1(\nu)$ which is the desired extension.

We refer the reader to  \cite[Appendix A]{MR4246976} for a proof that the spectral radius $\mathbb{L}$ is no larger than the spectral radius  of $\mathscr{L}_{\mathsf{M},f}$.

 \section{Stationary measures and the  Breiman Ergodic Theorem}

In this section we proof a result which can be seen  as a generalization  of  Breiman's Ergodic theorem  presented in \cite{Breiman}. The generalization remains on the fact that we do not require from the process {\it unique ergodicity.} More precisely, we allow our process to  have more than one stationary measure, and we still ensure that the empirical measures will converge to some stationary measure with probability one. The proof is a direct consequence of Furstenberg \cite[Lemma 2.5]{MR163345}.

Bellow we state this  lemma and for the sake of completeness we provide its demonstration.

\begin{lemma}\label{lema-furstemberg}
 Let $X$ be a standard Borel space  and $\mathsf{M}=(m_x)_{x\in X}$ a transition probability kernel on  $X$ having the Feller property. Let $P:C_b(X, \mathbb{R})\to C_b(X, \mathbb{R})$ be the associated Feller operator. If  $\Phi=(\Phi_{n})$ is a Markov process with transition probability kernel  $\mathsf{M}=(m_x)_{x\in X}$, then, for every $\varphi\in C_b(X, \mathbb{R})$ we have
\[
\lim_{n\to \infty}\frac{1}{n} \sum_{k=0}^{n-1} \big(P\varphi(\Phi_{k}) -\varphi(\Phi_k)\big) \rightarrow 0  \,\,\,\,  \qquad \mbox{a.e.}
\]

\end{lemma}

\begin{proof}
First, observe that 
\[
\displaystyle\frac{1}{n}\sum_{k=0}^{n-1} \big(P \varphi\left(\Phi_{k}\right)-\varphi\left(\Phi_{k}\right)\big)=\displaystyle\frac{1}{n}\sum_{k=0}^{n-1} \big(P \varphi\left(\Phi_{k}\right)-\varphi\left(\Phi_{k+1}\right)\big)
+\frac{\varphi(\Phi_{n})-\varphi(\Phi_{0})}{n}.
\]
 Since $\varphi$ is bounded, the sequence $(\varphi(\Phi_{n})-\varphi(\Phi_{0}))/n$ converges to $0$ as $n \to \infty$. Thus, to prove our claim we  will prove that the sequence of random variables  $\frac{1}{n}\sum_{k=0}^{n-1} \big(P \varphi\left(\Phi_{k}\right)-\varphi\left(\Phi_{k+1}\right)\big)$ converges a.e. to $0$ as $n\to \infty$.   
To this end, by the Kronecker's lemma, it is sufficient to prove that the sequence of random variables $(W_n)$ defined by
 \[
W_{n}=\sum_{k=0}^{n-1}\dfrac{P\varphi(\Phi_{k})-\varphi(\Phi_{k+1})}{k+1}
\] 
converges a.e. to a finite limit as $n\to \infty$. This is proved by showing that $(W_{n})$ is a martingale. Indeed,
\begin{multline*}
\mathbb{E}\left(W_{n+1} \mid \Phi_{1}, \cdots, \Phi_{n}\right) =W_{n}+\frac{1}{n+1} \mathbb{E}\left(P \varphi\left(\Phi_{n}\right)-\varphi\left(\Phi_{n+1}\right) \mid \Phi_{1}, \cdots, \Phi_{n}\right) 
\\
 =W_{n}+\frac{1}{n+1} P\varphi\left(\Phi_{n}\right)-\frac{1}{n+1} \mathbb{E}\left(\varphi\left(\Phi_{n+1}\right) \mid \Phi_{1}, \cdots, \Phi_{n}\right).
\end{multline*}
It follows from the definition of the Feller operator $P$ that $\mathbb{E}\left(\varphi\left(\Phi_{n+1}\right) \mid \Phi_{1}, \cdots, \Phi_{n}\right)=P \varphi\left(\Phi_{n}\right)$, which implies that $\mathbb{E}\left(W_{n+1} \mid \Phi_{1}, \cdots, \Phi_{n}\right)= W_{n}$. Therefore, $W_{n}$ is a martingale. Since $\varphi$ is bounded, 
\[
 \mathbb{E}\left(W_{n+1}^{2}\right) \leqslant 4 \|\varphi\|_\infty^{2} \sum_{k=1}^{\infty}\frac{1}{ k^{2}}.
 \]
  It follows from the Martingale Convergence Theorem that $W_{n}$ converges a.e. to a limit $W$ with $\mathbb{E}|W|< \infty$. 
\end{proof}

The next  can be seen  as a generalization  of  Breiman's Ergodic theorem  presented in \cite{Breiman}.

\begin{proposition}\label{krilov_}
Let $X$ be a compact metric space and $\mathsf{M}=(m_x)_{x\in X}$ a Feller transition probability on $X$. If $\Phi=(\Phi_{n})$ is a Markov 
chain with transition probability kernel  $\mathsf{M}=(m_x)_{x\in X}$, then with probability $1$, any accumulation point (in  the weak-$*$-topology) of the sequence 
\[
\frac{1}{n}\sum_{i=0}^{n-1}\delta_{\Phi_{i}}
\]
is a stationary measure for $\Phi$.
\end{proposition} 

\begin{proof}
Since $X$ is compact  $C_b(X, \mathbb{R})$ is separable, then it follows from \ref{lema-furstemberg} that with probability $1$
\begin{equation}\label{krilov1}
\lim_{n\to \infty}\frac{1}{n}\sum_{i=0}^{n-1} P\varphi(\Phi_{i})-\varphi(\Phi_{i})=0
\end{equation}
for every $\varphi\in  C_b(X, \mathbb{R})$.
Let $(\Phi_{i})$ be a sample sequence for which this is the case. 

Assume that $\nu$ is an accumulation point 
of $\frac{1}{n}\sum_{i=0}^{n-1}\delta_{\Phi_{i}}$, that is, there is an increasing subsequence $(n_{k})_{k\in\mathbb{N}}$ such that 
\begin{equation}\label{krilov2}
\lim_{k\to \infty}\frac{1}{n_{k}}\sum_{i=0}^{n_{k}-1}\delta_{\Phi_{i}}= \nu
\end{equation}
in the weak-$*$-topology. Since $P$ is Feller, we have that $P\varphi-\varphi$ is continuous for every continuous function $\varphi$. In particular, it follows from \eqref{krilov2} that 
\[
\lim_{k\to \infty}\frac{1}{n_{k}}\sum_{i=0}^{n_{k}-1}P\varphi(\Phi_{i})-\varphi(\Phi_{i})= \int_X (P \varphi-\varphi)\, d\nu
\]
for every continuous function $\varphi$. Now, it follows from \eqref{krilov1} that 
\[
\int_X P\varphi\, d\nu=\int_X \varphi\, d\nu
\]
for every continuous function $\varphi$. This shows that $\nu$ is a stationary measure.
\end{proof}

 \subsubsection{Application}
 The previous result has an immediate application in the transfer operator scenario.
 
 \begin{corollary}\label{Breiman-Ruelle}
Let  $\mathsf{M}=(m_x)_{x\in X}$ be a  transition probability kernel on a compact space $X$ satisfying the Feller property,  and  consider  $f\in C_b(X, \mathbb{R})$  a bounded continuous normalized potential. If  $\mathscr{L}_{\mathsf{M},f}$  is  the associated transfer operator, denote 
 $\widetilde{\mathsf{M}}=(\widetilde{m}_x)_{x\in X}$ the transition probability kernel given by,
$
 \widetilde{m}_x(A)=\mathscr{L}_{\mathsf{M},f}(\mathbf{1}_A)(x)
 $
  and let  $\Psi=(\Psi_n)$ denote its associated Markov chain starting at $x$. Then with probability $1$,  any accumulation point of the sequence 
  \begin{equation}\label{accumulation-point}
\frac{1}{n}\sum_{i=0}^{n-1}\delta_{\Psi_{i}} 
  \end{equation}
    is a stationary measure  of $\Psi$. In particular a such limit  $\nu$ must satisfies $\mathscr{L}_{\mathsf{M},f}^*\nu=\nu$, i.e, $\nu$ is a fixed point of the dual of the transfer operator.
 \end{corollary}
 \begin{proof}
 	The proof follows immediately from  observing that $X$ is compact and that  $\widetilde{\mathsf{M}}=(\widetilde{m}_x)_{x\in X}$ has the Feller property, then we apply Proposition \ref{krilov_} and the result follows.
 \end{proof}

 A particular interesting case is when $\mathsf{M}=(\mu\times \delta_x)_{x\in X}$. In this case $\mathscr{L}_{\mathsf{M},f}$ is the Standard Ruelle operator, and therefore, any accumulation point of (\ref{accumulation-point}) is a {\it Gibbs measure}.

 \section{A FCLT for a Long-Range $O(N)$ model via Poisson  Equation}\label{CLT}

In this section we study the validity of a Functional Central Limit Theorem 
in the setting of Statistical Mechanics on the one-dimensional lattice. We use the version obtained here to prove FCLT  for a long-range $O(N)$ model with polynomial decay of correlations. This is, as far as we know, a new result in Statistical Mechanics. 
That task would require an extension of the main result of \cite{kloeckner_2020} (see also \cite{MR4637163})
to the setting of Example \ref{Ruelle-padrao} where uncountable alphabets are considered.

We will consider long-range $O(N)$ models
on the $\mathbb{N}$ lattice.  In the $O(N)$, for $N\geqslant 2$, the fibers are uncountable and given by $E= \mathbb{S}^{N-1}$, the unit sphere in the $N$-dimensional Euclidean space.
The potential we will consider is a Dyson type potential   given by the expression
$
f(x) = \sum_{n=1}^{\infty} J(n)\langle x_1,x_{n+1}\rangle,
$
where $J(n)=O(n^{-2-\varepsilon})$ and the scalar multiplication 
is replaced by the usual inner product of $\mathbb{R}^n$.

In Section \ref{CLT-NON-COMP} we study the validity of a FCLT in a non compact alphabet setting with respect to an adapted Dyson type potential.

\subsection{Proving a FCLT by solving the Poisson equation}
The problem of proving a  Central Limit Theorem  for a Markov process often relies on verifying some analytical condition on the associated transfer operator, see for instance \cite{BhLe,Gordin,MR1862393,Wood}.
Most commonly, when people want to prove a Central Limit Theorem  in the context of transfer operators, they often appeal to the Nagaev-Guivarc'h method.   An  abstract version of this method can be stated as follows.

Let $X_1, X_1, \ldots $ be a sequence of real random variables. Denote $S_n=X_1+\cdots+X_n$. Assume that there is $\delta>0$ and functions $c(t), \lambda(t)$  and $d_n(t),$ defined on $[-\delta, \delta]$ such that, 
\begin{equation}\label{nagaev}
\mathbb{E}(e^{itS_n})=c(t)\lambda(t)^n+d_n(t), \quad \forall t\in [-\delta, \delta], n\in \mathbb{N}.
\end{equation}

Assume that we can write $\lambda(t)=\exp(iAt -\sigma^2/2+ o(t^2))$, for some constant $A$ and $\sigma^2$ in $\mathbb{C}$  for  small values of $t$. Also suppose that the function $c(t)$ continuous on $0$, and that the $\|d_n\|_\infty$ goes to zero when $n\to \infty$. Then $A\in \mathbb{R}, \sigma^2\geqslant 0$ and  $(S_n-nA)/\sqrt{n}$ converges to a Gaussian  distribution $\mathscr{N}(0, \sigma^2)$.  The proof is  essentially the same as Central Limit Theorem with some minor changes (to more details see \cite{MR3309098}).

In the transfer operator context, writing  $\mathbb{E}(e^{itS_n})$ as something like (\ref{nagaev}) will pass by finding a suitable subspace where the transfer operator restricted to it becomes {\it quasi-compact}, i.e, can be decomposed as $\Pi+Q$ where  spectrum of $Q$ is contained in  a disc of radius  $r<1$, and $\Pi$ has finite rank.
Indeed, that is the main challenge of the method, finding  a subspace on which the 
transfer operator exhibits the  {\it spectral gap} property.
In general, the absence of spectral gap will occur when we have super-polynomial decay of correlations.  In \cite{MR3538412}  considering the Ruelle operator in the Walters space a  potential where the associated Ruelle operator has no spectral gap is exhibited.

 Here we will follow the  method presented in \cite{BhLe} which consists in to find  a solution of  the {\it Poisson equation}.  
  Namely, assume that $(Y_{n})_{n\in\mathbb{N}}$ is an ergodic stationary Markov
  chain whose stationary distribution is $\mu$.
  Note that $P$ takes $L^{2}(\mu)$ into $L^{2}(\mu)$.
  Given an almost surely non-constant observable
  $\phi\colon S\to \mathbb{R}$ such that  $\phi\in L^{2}(\mu)$ and $\int_{X} \phi\, d\mu=0$,
  consider the \emph{Poisson equation}
  \[
  (I-P)\upsilon=\phi.
  \]
  Assume that there is a solution $\upsilon\in L^{2}(\mu)$ and  let
  $\varrho\equiv \mu(\upsilon^2)-\mu(P\upsilon)^2>0$ (ergodicity implies $\varrho>0$).
  Consider the stochastic process $Y_{n}(t)$ given by
  \[
  Y_{n}(t)=\displaystyle\frac{1}{\varrho\sqrt{n}}\sum_{j=0}^{[nt]}\phi(Z_{j}), \qquad 0\leqslant t<\infty
  \]
  taking values in the space \( D[0,\infty) \) of real-valued right continuous
  functions on \( [0,\infty) \) having left limits, endowed with the Skorohod topology.
  Then, the process $ Y_{n}(t)$ converges in distribution
  (weak-$*$ convergence) to the Wiener measure on \( D[0,\infty) \).
    
    The above  FCLT  is obtained by reducing the problem to the martingale case.  See  \cite{BhLe,Gordin} for this reduction 
 and Billingsley \cite[Theorem 18.3]{MR1700749} for a FCLT for martingale differences.

  Let $\mathsf{M}=(m_x)_{x\in X}$ be a transition probability kernel, $f$ a normalized potential (i.e., $\mathscr{L}_{\mathsf{M}, f}\mathbf{1}=\mathbf{1}$) and $\mu$ a stationary measure. Consider the transition probability $\widetilde{m}$ given by
 \[
\widetilde{m}(x,A)=\mathbb{L}(\mathbf{1}_A)(x), 
 \]
 where the operator $\mathbb{L}$ is the extension of the operator $\mathscr{L}_{\mathsf{M}, f}$ to ${L}^2(\mu)$.
It is readily checked that the induced Feller  operator satisfies 
$P\varphi (x)=\mathbb{L}(\varphi)(x)$, and also that any probability  measure satisfying $\mathscr{L}_{\mathsf{M}, f}^*\mu=\mu$
is an stationary measure for $P$.


Denote by $\Psi=(\Psi_n)$ the Markov chain taking values in $X$ defined by $(\widetilde{m}_x)_{x\in X}$.
In view of  Theorem  \ref{lema-eq-distribution}  the above result in our setting take the following form.

\begin{theorem}\label{functionalclt}
	Let $P$ be the transfer operator induced by the extension $\mathbb{L}$
	associated with a continuous
	and normalized potential and $\mu\in \mathscr{G}^{*}$.
	Let $\phi:X\to \mathbb{R}$ be a non-constant observable  in $L^{2}(\mu)$
	satisfying $\mu(\phi)=0$. If
	there exists a solution $\upsilon\in L^{2}(\mu)$ for Poisson's equation $(I-\mathbb{L})\upsilon=\phi$,
	then the stochastic process $Y_{n}(t)$, given by
	\begin{align}\label{eq-Ynt}
	Y_{n}(t)=\displaystyle\frac{1}{\varrho\sqrt{n}}\sum_{j=0}^{[nt]}\phi\circ \sigma^{j}, \qquad 0\leqslant t<\infty,
	\end{align}
	where $\varrho=\mu(\upsilon^2)-\mu(P\upsilon)^2$, converges in distribution to the Wiener measure in
	$D[0,\infty)$.
\end{theorem}

 Before we continue we need to introduce some framework of  Optimal Transport Theory.

 \subsection{Optimal transport}
 Define the Wasserstein metric with respect to distances of the form $\omega\circ d$,
 where $\omega$ is a modulus of continuity, by putting 
 \[
 W_\omega(\mu, \nu)=\inf_{\Pi\in \Gamma(\mu,\nu)}\int_{X\times X}\omega\circ d(x,y) \;d\Pi(x,y)
 \]
 where $\Gamma(\mu, \nu)$ is the set of the Borel probability measures on $X\times X$ with marginals $\mu$ and $\nu$. We refer to an element of  $\Gamma(\mu, \nu)$ as a {\it transport plan.}
 A transport plan attaining the above infimum is called an \emph{optimal transport plan} with respect to the \emph{cost} $\omega\circ d$. It is worth mentioning that optimal transport plans always exist, and that $W_\omega$ metrizes the weak-$*$-topology on $\mathscr{P}(X)$. For more details see \cite{MR2459454}. 
 
 
 \begin{definition}
 	
 A non-negative function $F:\mathbb{N}\times (0,R)\to (0, \infty)$ is said to be a decay function if
 \begin{itemize}
 	\item[a)] $F$ is non increasing in $n$ and non-decreasing  and concave in the second variable; 
 	\item[b)] $\lim_{n\to \infty} F(n, r)=\lim_{r\to 0} F(n, r)=0$;
 	\item[c)] There exists $C>0$ such that, $\forall t, r$, we have $F(t,r)\leqslant cr$;
 	\item[d)] $\forall k, n\in \mathbb{N}$ and $r>0$, we have $F(k+n, r)\leqslant F(k, F(n,r)).$
 \end{itemize}
  \end{definition}

Let  $\mathscr{T} $  be a linear operator defined on the finite signed measure space preserving the
 	probability measures space $\mathscr{P}(X).$ We say that $\mathscr{T}$ has $\omega$-{\it decay rate} $F$ if for all $n\in \mathbb{N}, \mu, \nu\in \mathscr{P}(X) $ we have that
 	\[
 	W_\omega(\mathscr{T}\mu, \mathscr{T}\nu)\leqslant F(t,	W_\omega(\mu, \nu) ).
 	\]
 where $F$ is a decay function. In particular, when $F$ is given by 
  \[
  F(n,r)= \frac{Br}{(n r^\alpha+b)^{\frac{1}{\alpha}}}
  \]
for some $ B\geqslant 1,b\in(0,1)$ and for some $ n, r$, we say that $\mathscr{T}$ has polyomial $\omega$-decay.

%
%

 \subsection{Coupling and flat potentials}
 
 A \emph{coupling} of a transition probability kernel  $\mathsf{M}= (m_x )_{x\in X}$ is a family $\mathsf{P}$ 
 of
 probability measures $\Pi_{x,y}^n \in \Gamma(m^n_x , m^n_y )$,  $n\in \mathbb{N} $  and $x, y \in X$, such that, for
 each $n$, the map $(x, y) \mapsto \Pi^n_{x,y}$ is Borel-measurable.

 A coupling $\mathsf{P} = (\Pi_{x,y}^n)_{n,x,y}$ of the
 transition probability  kernel $\mathsf{M} =(m_x )_{x\in X}$ is said to be $\omega$-{\emph Hölder} if there exists a constant $C$ such 
 that for all $n\in \mathbb{N}$ and $x,y\in X$ we have the following:
 \begin{equation}\label{Holder-coupling}
 \int \omega\circ d(x_n,y_n)\; d\Pi_{x,y}^n(\bar x,\bar y)\leqslant C \omega\circ d(x,y).
 \end{equation}

Given a  coupling  $\mathsf{P}$  and a normalized potential $f$, we define  the family $\mathsf{P}_f=(e^{f^n(\bar{x})}\Pi_{x,y}^n(\bar{x}, \bar{y}))_{n,x,y},$
and we say that $\mathsf{P}_f$ has {\it $\omega$-decay rate} $F$ when for some constant $C$ (depending on $f$),  for every $n$ and all $x,y\in X $ it holds,
 \[ \int \omega\circ d(x_n,y_n)e^{f^n(\bar{x})} \,d\Pi^n_{x,y}(\bar x,\bar y) \leqslant C F(n,\omega\circ d(x,y)).\]

If a coupling $\mathsf{P}$ has $\omega$-decay rate of type a) we say that $\mathsf{P}$ has exponential decay, and if $\mathsf{P}$ has $\omega$-decay rate of type b) we say that $\mathsf{P}$ polynomial decay.

 \begin{definition}[Flatness Condition]
 	Fix a coupling $\mathsf{P}=(\Pi^n_{x,y})_{n,x,y}$ associated to a transition probability  kernel $\mathsf{M}$. We say that a potential 
 	$f\in C^\omega(X, \mathbb{R})$ is  \emph{flat } with respect to $\mathsf{P}$ if there is some positive  constant $C$ such that 
 	for all $n\in \mathbb{N}$, all  $x,y\in X$  and $\Pi_{x,y}^n$-almost $(\bar x, \bar y)$ we have 
 	the following:
 	\begin{equation}
 	|f^n(\bar x)-f^n(\bar y)|\leqslant C\omega\circ d(x,y).
 	\end{equation}
 \end{definition}

  The next theorem plays a fundamental role in the next section. It is part of the main result of 
     \cite{kloeckner_2020} (Theorem 4.1, see also \cite{MR4637163}) and we state sake of clarity.



\begin{theorem}\label{Kloeckner-main-result}
Let $\mathsf{M}=(m_x)_{x\in X}$ be a transition kernel on a compact metric space $X$  and $\omega$ be a modulus of continuity.
Let $f\in C^{\omega}(X, \mathbb{R})$ be flat and normalized potential with respect to $\mathscr{L}\equiv\mathscr{L}_{\mathsf{M}, f}$.
	Suppose that $\mathsf{M}=(m_x)_{x\in X}$  has a coupling $\mathsf{P}$  such that $\mathsf{P}_f$ has $\omega$-decay  function $F$.
	
	Then there exist a constant $C>0$ 
	 such that for all $\mu, \nu \in \mathscr{P}(X)$ 
\[
W_\omega\big(\mathscr{L}^{* n}\mu, \mathscr{L}^{* n}\nu\big)\leqslant CW_\omega(\mu, \nu)
\]	
In particular, when $F$ is polynomial then  $\mathscr{L}^{*}_{\mathsf{M}, f}$ decays  polynomially in the Wasserstein metric with the same degree.
\end{theorem}

\subsection{Decay rate on uncountable alphabets}\label{Decay-Rate-Uncout}
In this section $X$ will stand for $E^{\mathbb{N}}$, where $E$ is   a standard Borel space. Our goal in this section is to adapt the notion of {\it natural coupling} introduced in  \cite[Definition 5.2]{kloeckner_2020} to contemplate the context of Example \ref{Ruelle-padrao}, where   general alphabets are considered, and then show that for a suitable choice of potential this coupling has polynomial decay rate.


By considering the transition kernel $\mathsf{M}=(\mu\times \delta_x)_{x\in X}$   given in Example \ref{Ruelle-padrao},  and a potential   $f\in C^{\omega}(X, \mathbb{R})$  we have  the standard Ruelle operator
\[
\mathscr{L}_{\mathsf{M}, f}\varphi(x)=\int_{E}e^{f(ax)} \varphi(ax)\, d\mu(a).
\]
 For each   fixed $w=(r_1, r_2, \dots, r_n)\in E^n$ and $x,y\in X$ define $\bar{x}^{w}_n$ and $\bar{y}^{w}_n$ in $X^n$ by specifying its projections for $j=1, \ldots, n$ as 
\[
\Pi_j(\bar{x}^{w}_n)=\left(r_j, r_{j-1}, \ldots, r_1, x_1, x_2, \cdots\right )
\]
and 
\[
\Pi_j(\bar{y}^{w}_n)=\left(r_j, r_{j-1}, \ldots, r_1, y_1, y_2, \cdots\right ).
\]
Consider $\mathsf{P}=(\Pi_{x,y}^n)_{n,x,y}$  defined by
\begin{equation}\label{natural-coupling}
\Pi^n_{x,y}=\int_{E^n}\delta_{(\bar{x}^{w}_n , \bar{y}^{w}_n)}\,d\mu^n(w).
\end{equation}
More explicitly, if $\varphi:X\times X\to \mathbb{R}$ is a measurable function then, the expression (\ref{natural-coupling}) means that, 
\[
\int_{X\times X} \varphi(x,y)\,d\Pi_{x,y}=\int_{E^n} \varphi(\bar{x}^{w}_n , \bar{y}^{w}_n) \, d\mu(w).
\]

One can easily shows that 
$\mathsf{P}=(\Pi_{x,y}^n)_{n,x,y}$ 
is a coupling for the transition probability kernel $\mathsf{M}= (\mu\times \delta_x )_{x\in X}$, and it is called {\it natural coupling.}
It is worth observing that   for $(\Pi^n_{x,y})_{n,x,y}$ almost all $(\bar x, \bar y)$  in $X^{n}\times X^{n}$  we have the following {\it contraction property,  }
\begin{align*}
 d(x_n,y_n)\leqslant \dfrac{1}{2^n}d(x,y).
\end{align*}
In fact, by its very nature $\mathsf{P}=(\Pi_{x,y}^n)_{n,x,y}$ only sees the paired orbits $r_jr_{j-1}\cdots r_1 x$ and  $r_jr_{j-1}\cdots r_1 y$, then we have
\begin{align*}
d(x_n, y_n)
&=d(r_nr_{n-1}\cdots r_1 x,r_nr_{n-1}\cdots r_1 y)
\\
&\leqslant
\dfrac{1}{2^n}d(x,y).
\end{align*}
\begin{proposition}\label{decay-rate}
Let $\mathsf{M}=(\mu\times \delta_x)_{x\in X}$  be a transition probability  kernel defined  on a standard Borel
 space $X$. Consider  $\mathsf{P}=(\Pi^n_{x,y})_{x,y,n}$ the natural coupling defined in (\ref{natural-coupling}). For each  $\varepsilon>0$  consider
$\omega_\varepsilon(r)=\log(r_0/r)^{-\varepsilon}$ and $f\in C^{\varepsilon \log}(X, \mathbb{R})$ a normalized potential. Then $\mathsf{P}_f$ has polynomial decay of degree $\varepsilon$ with respect $\omega_\varepsilon$. 
\end{proposition}
\begin{proof}
	The proof will closely follow the one given in \cite{kloeckner_2020}, Lemma 5.3, with some small modifications.
	Observe  that, 
\begin{multline*}
\int\omega_\varepsilon \circ d(x_n,y_n)\, d \big(e^{f^n(\bar{x})}\Pi^n_{x,y}\big)(\bar{x}, \bar{y})
\\\leqslant
\omega_{\varepsilon}\left(\int d(x_n ,y_n)\, d \big(e^{f^n(\bar{x})}\Pi^n_{x,y}\big)(\bar{x}, \bar{y})\right)
\leqslant \omega_\varepsilon\left(2^{-n}d(x,y)\right).
\end{multline*}
On the other hand a straightforward calculation  shows that for a suitable choice of $\lambda^\prime,$
\[
\omega_\varepsilon\left(2^{-n}d(x,y)\right)=\dfrac{1}{(\omega_\epsilon(d(x,y))^{1/\varepsilon}+ t\log \lambda^\prime)^\varepsilon},
\]
from which follows the result.
\end{proof}

\begin{theorem}\label{uncountable-version}
Suppose $X=E^{\mathbb{N}}$ where $E$ is a compact metric space.    Let $\mathsf{M}=(\mu\times \delta_x)_{x\in X}$ be a transition probability kernel and consider $f\in C^{\varepsilon \log}(X, \mathbb{R})$  a flat potential with respect to the moduli of continuity  $\omega_\varepsilon(r)=\log(r_0/r)^{-\varepsilon}$ and the natural coupling $\Pi^n_{x,y}=\int_{E^n}\delta_{(\bar{x}^{w}_n , \bar{y}^{w}_n)}\,d\mu^n(w)$. Then the transfer operator $\mathscr{L}_{\mathsf{M}, f}$ has a strictly positive eigenfunction $h_f\in  C^{\varepsilon \log}(X, \mathbb{R})$ with respect to some positive eigenvalue $\lambda_f$ and its dual $\mathscr{L}_{\mathsf{M}, f}^*$ has a conformal measure $\nu_f$ with respect to the eigenvalue $\lambda_f.$ The measure $\mu_f=h_f d\nu_f$ is a stationary measure. 

If $\bar{f}= f+\log h_f -\log h\circ \sigma -\log \lambda_f$ is the normalized potential,  and  if $\varphi$ is an observable satisfying $\int \varphi d\mu_{\bar{f}}=0,$ we must have that, 
$
\|\mathscr{L}_{\mathsf{M},\bar{f}}^n\varphi\|\leqslant C\frac{\|f\|_{\omega_\varepsilon}}{n^\varepsilon}.
$
In particular the one sided shift has polynomial  decay of correlations.
\end{theorem}

\begin{proof}
	The proof  theorem relies in Theorems 3.2, 4.1 and 5.8 of \cite{kloeckner_2020}. 
The first part of theorem is an immediate application of Theorem 3.2.
For the second part, it was showed in Proposition \ref{decay-rate}, that   the natural coupling  $\mathsf{P}=(\Pi^n_{x,y})_{n,x,y}$ has polynomial decay rate with respect to $\omega_\varepsilon$, therefore we can apply  Theorem 4.1 of \cite{kloeckner_2020} (see Theorem \ref{Kloeckner-main-result}) obtaining that $\mathscr{L}^*_{\mathsf{M}, f}$ also  has polynomial decay rate.
	Minor adaptations in the proof of Theorem 5.8 of \cite{kloeckner_2020} provide the desired polynomial decay of correlations.
\end{proof}

\subsection{FCLT for a  long-range $O(N)$ model}\label{O(N)-model}
Let $X=(\mathbb{S}^N)^{\mathbb{N}}$ where $\mathbb{S}^N$ is  the $(N+1)$-dimensional euclidean sphere. Let $\mathbb{\mu}$ be the Lebesgue measure on $\mathbb{S}^{N}.$ Consider the Dyson type potential, 
\[
f(x) = \sum_{n=1}^{\infty} \dfrac{\langle x_1,x_{n+1}\rangle}{n^{2+\varepsilon}}.
\]
where $\langle x,y\rangle$ means the inner product of $\mathbb{R}^n$.

Below we will show that this model has a polynomial decay of correlations. To do that we need two technical  lemmas.
\begin{lemma}
	For each $\varepsilon>0$ we have that $f\in C^{\varepsilon \log}(X, \mathbb{R})$, where $\omega_\varepsilon(r)=\log(r_0/r)^{-\varepsilon}$.
\end{lemma}
\begin{proof}
	The proof of the lemma will relies  in two reductions. The first one is to observe that since $C^{\varepsilon \log}(X, \mathbb{R})$ is a Banach algebra it is sufficient to proof that each $\langle x_1, x_{n+1} \rangle$ lives in $C^{\varepsilon \log}(X, \mathbb{R})$.

	Secondly,  by observing that $\langle x_1, x_{n+1} \rangle= \sum_{k=1}^N x_1^kx_{n+1}^k$,  the same reasoning above ensures that we only need to verify that each map $\pi_{k,n}: (\mathbb{S}^N)^{\mathbb{N}}\to [0,1]$ given by $\pi_{k,n}(x)=x_n^k$, $n\in \mathbb{N}$, $k=1, \ldots, N$  is in $C^{\varepsilon \log}(X, \mathbb{R})$.

Now observe  that there exists a constant $D$ such that  for each  $t\in [0,1]$ we have that $t\leqslant D \omega_{\varepsilon}(t)$. Indeed, it suffices to observe that  for each $\varepsilon> 0$, the limit  $\lim_{t\to 0}t |\log(t)|^\varepsilon$ is finite. Then, we have that
	\begin{eqnarray}
	|\pi_{k,n}(x) -\pi_{k,n}(y)|=|x_n^k-y_n^k|\leqslant 2 d(x,y)\leqslant 2D \omega(d(x,y)).
	\end{eqnarray}
	This shows that $\pi_{k,n}\in C^{\varepsilon \log}(X, \mathbb{R})$.
\end{proof}

\begin{lemma}\label{Lemma-flat-1}
Suppose that  $f\in C^{\varepsilon \log}(X, \mathbb{R})$ for $\varepsilon>2$, then   $f\in C^{(\varepsilon-1)\log}(X, \mathbb{R})$. In particular the Dyson type potential is flat with respect to the natural coupling and ${\omega}_{\varepsilon-1}$.
\end{lemma}
\begin{proof}
	Observe that  $d(a_1\cdots a_j x, a_1\cdots a_j y)=2^{-j}d(x,y)$.
	This implies, for $0\leqslant j\leqslant n-1$, that
$
	|f(\sigma^j(a_1\cdots a_nx)) - f(\sigma^j(a_1\cdots a_ny))|
	\leqslant \omega (2^{n-j}d(x,y)).
$
Therefore,
\[
	|\sum_{j=0}^{n-1}f(\sigma^j(a_1\cdots a_nx))-f(\sigma^j(a_1\cdots a_ny))|\leqslant
	\sum_{j=0}^{n-1}\omega (2^{-j}d(x,y)).
\]
On the other hand, 
	\[
	\omega(2^{-j}d(x,y))
=\Big[ j\log 2+\log\left(\dfrac{r_0}{d(x,y)}\right)\Big]^{-\varepsilon}.
	\]
	Thus,
	\begin{multline*}
	\sum_{j=0}^{n-1}\omega (2^{-j}d(x,y))
	 \leqslant \sum_{j=0}^{\infty}\omega
	(2^{-j}d(x,y))=\sum_{j=0}^{\infty}\left[ j\log 2+\log\left(\dfrac{r_0}{d(x,y)}\right)\right]^{-\varepsilon}
	\\[0.3cm]
	\leqslant  \lim_{\delta\to 0}\displaystyle\int_{\delta}^\infty \left[j\log 2+\log\left(\dfrac{r_0}{d(x,y)}\right)\right]^\varepsilon \, dx
     =
	C \left[\log\left(\dfrac{r_0}{d(x,y)}\right)\right]^{-(\varepsilon-1)}.
	\end{multline*}
	Showing that
	\[
	\sum_{j=0}^{n-1}\omega (2^{-j}d(x,y))\leqslant C \tilde{\omega}(d(x,y)),
	\]
	where $\tilde{\omega}(r)=\log(r_0/r)^{-(\varepsilon-1)}$. This
	proves the existence of a constant $C>0$ such that, for every $n\in \mathbb{N}$,
	\begin{equation}\label{flat}
	\big|\sum_{j=0}^{n-1}f(\sigma^j(a_1\cdots a_nx))-f(\sigma^j(a_1\cdots a_ny))\big|\leqslant C\tilde{\omega} (d(x,y)).
	\end{equation}
	As a consequence of the above inequality we get that $f$ is flat with respect to the modulus of continuity $\omega_{\varepsilon -1}$ and that $f\in
	C^{(\varepsilon-1) \log}(X,\mathbb{R})$.
	\end{proof}

 Therefore, by applying Theorem   3.2  of \cite{kloeckner_2020}
 we get a strictly positive eigenfunction $h$ living in $C^{(\varepsilon -1)\log}(X, \mathbb{R})$ with associated eigenvalue $\lambda>0$. This allow us to consider the normalized  potential $\bar{f}= f+\log h -\log h\circ \sigma -\log \lambda$
\begin{lemma}\label{Lemma-flat-2}
The normalized  potential $\bar{f}= f+\log h -\log h\circ \sigma -\log \lambda$ is also a flat potential. 
\end{lemma}
\begin{proof}
	See discussion  in \cite[page 22]{kloeckner_2020}.
\end{proof}

Consider  $\mathsf{P}=(\Pi^n_{x,y})_{x,y,n}$ the natural coupling defined in (\ref{natural-coupling}), by Proposition \ref{decay-rate} we get that $\mathsf{P}_f$ has polynomial decay rate with respect to $\omega_{\varepsilon -1}$.
Therefore, we can apply Theorem \ref{uncountable-version} to get a constant $C>0$  such that for any $\varphi\in C^{(\varepsilon-1)\log}$ satisfying $\int_X \varphi \, d\mu_f=0$ we have that 
\[
\|\mathscr{L}_{\mathsf{M},\bar{f}}^n\varphi\|\leqslant C\frac{\|f\|_{\omega_\varepsilon}}{n^{\varepsilon-1}}.
\]
Since we are considering $\varepsilon>0$ we have $\sum_{n=2}^\infty \|\mathscr{L}_{\mathsf{M},\bar{f}}^n\varphi\|\leqslant \sum_{n=2}^\infty C n^{-(\varepsilon-1)}<\infty$, thus implying that $v=-\sum_{n=0}^\infty \mathscr{L}_{\mathsf{M},\bar{f}}^n\varphi$ is a well defined element of $L^2(\mu_f)$. One can easily verify that $v$ is a solution for Poisson's equation, allowing us to apply Theorem \ref{functionalclt}, thus concluding the FCLT to this model.

\subsection{Polinomial decay rate in a non-compact setting}\label{CLT-NON-COMP}
It is useful for  our purposes to observe  that the proof of Theorem \ref{Kloeckner-main-result} in \cite{kloeckner_2020} (see also \cite{MR4637163})  uses the compactness of $X$ for two reasons.
The first reason  is to ensure that $\text{diam}(X)$ is finite. The second reason is to obtain a {\it strictly positive} eigenfunction $h \in C^{\omega}(X, \mathbb{R}),$  whose its   associated eigenvalue  is the spectral radius $\mathscr{L}_{\mathsf{M}, f}$, in order to construct a normalized potential.
 The remaining of the proof is based  on results from Optimal Transport Theory in which compactness does not play any role.
	                                                                       
Therefore if we would like to provide a version of Theorem (\ref{Kloeckner-main-result}) when $(X, \mathscr{X})$ is non-compact we need to arrange two things. Firstly we need  to endow $X$ with an equivalent bounded metric, and secondly to provide a strictly positive eigenfunction $h \in C^{\omega}(X, \mathbb{R})$.
\subsubsection{A bounded metric}
Suppose $d$ is a metric on $X$, if $d$ is not bounded we can change $d$ by the  metric  $\tilde{d}=\min\{1, d\}$. From the topological perspective the metrics  $d$ and $\tilde{d}$ generate the same topology on $X$. It is easy  to verify that this new metric space  $(X, \tilde{d})$ is now  bounded and still  complete. 
On the other hand, it is important to note that other metric features  will change. For instance,  the space of the $\omega$-Hölder continuous functions with respect to this new metric will  be different.                                             
  \subsubsection{Providing the positive eigenfunction}
   To obtain a strictly positive eigenfunction, we can follow a classical approach by using a general form of the Arzelà-Ascoli theorem, see \cite[Theorem 6.4]{Dugundji}. Let us provide a brief overview of this construction for completeness.

 Bellow   $\mathsf{M}=(m_x)_{x\in  X}$ is  a transition probability kernel on a  standard Borel space  $(X, \mathscr{X})$ equipped with a bounded metric, $f$ is a flat potential  belonging to  $C^{\omega}(X, \mathbb{R})$ for some modulus of continuity $\omega$, and $\mathscr{L}_{\mathsf{M}, f}$ refers to the associated transfer operator with respect to $\mathsf{M}$ and  $f$.
	   
	                                                                                                                Denote by  $\varrho$  the spectral radius of $\mathscr{L}_{\mathsf{M}, f}$, we start by showing that there exist positive  constants $A$ and $B$ such that for every $n\in \mathbb{N}$  we have that, 
	                                                                                                                \[
	                                                                                                                A\leqslant   \varrho^{-n}\mathscr{L}^n_{\mathsf{M}, f} \mathbf{1}\leqslant B.
	                                                                                                                \]	
	                                                                                                                This will follows from using the flatness property of the potential $f$. Indeed, since $f$ is  a flat potential  we have for every $n\in \mathbb{N}$ that,
	                                                                                                                \[
	                                                                                                                e^{f^n(\bar{x})}-e^{f^n(\bar{y})}= e^{f^n(\bar{y})}\big(e^{f^n(\bar{x})-f^n(\bar{y})}-1\big)\leqslant C \omega\circ d(x,y)
	                                                                                                                \]
	                                                                                                                then by using that $\Pi_{x,y}$ has marginals $m_x^n$ and $m_y^n$,
	                                                                                                                \begin{eqnarray}\label{positivity-1}
	                                                                                                                \mathscr{L}^n_{\mathsf{M}, f}\mathbf{1}(x)=\int e^{f^n(\bar{x})}\,d\Pi^n_{\bar{x},\bar{y}}(x, y)
	                                                                                    =\int e^{f^n(\bar{y})}e^{f^n(\bar{x})-f^n(\bar{y})}\,d\Pi^n_{\bar{x},\bar{y}}(x, y)
	                                                                                                                \\
	                                                                                                              \leqslant e^{C \text{diam}(X)}\int e^{f^n(\bar{y})}\,d\Pi^n_{{x},{y}}(\bar{x},\bar{y})\nonumber
	                                                                                                                =C^{\prime } \mathscr{L}^n_{\mathsf{M}, f}\mathbf{1}(y), 
	                                                                                                                \end{eqnarray} 
	                                                                                                                for all $x, y\in X$. Therefore by symmetry, we have for all $x, y \in X$ that,
	                                                                                                                \begin{eqnarray}\label{positivity-2}
	                                                                                                                \frac{1}{C^\prime}\mathscr{L}^n_{\mathsf{M}, f}\mathbf{1}(y)\leqslant \mathscr{L}^n_{\mathsf{M}, f}\mathbf{1}(x)\leqslant
	                                                                                                                C^\prime \mathscr{L}^n_{\mathsf{M}, f}\mathbf{1}(y).
	                                                                                                                \end{eqnarray}
	                                                                                                                The above inequality has as  a  consequence that  the number  
	                                                                                                                \begin{equation}\label{positivity-3}
	                                                                                                                \tilde{\varrho}=\limsup_{n\to \infty}\big(\mathscr{L}^n_{\mathsf{M}, f}\mathbf{1}(x)\big)^{1/n}
	                                                                                                                \end{equation}
	                                                                                                                is finite and non-zero, and it does not depend on $x\in X$.  Since $\mathscr{L}^n_{\mathsf{M}, f}$ preserve the cone of positive functions we get that $\varrho=\tilde{\varrho}$.
	                                                                                                                
	                                                                                                                A  second consequence of  (\ref{positivity-2}) is  that in view of  (\ref{positivity-3})  we can find  positive constants $A$ and $B$ such that for every $n\in \mathbb{N}$ and $x\in X$,  
	                                                                                                                \[
	                                                                                                                A\varrho^{n}\leqslant   \mathscr{L}^n_{\mathsf{M}, f} \mathbf{1}(x) \leqslant \varrho^{n}B.
	                                                                                                                \]

	                                                                                                                Now observe that (\ref{positivity-1}) also implies that the family $\{\mathscr{L}^n_{\mathsf{M}, f}\mathbf{1}\}_{n\in \mathbb{N}}$ is an equicontinuous sequence. Define   for each $n\in \mathbb{N}$ the function $h_n\in C^{\omega}(X, \mathbb{R})$ by putting, 
	                                                                                                                \[
	                                                                                                                h_n(x)=\dfrac{1}{n}\sum_{j=0}^n \varrho^{-j}\mathscr{L}^j_{\mathsf{M}, f} \mathbf{1}(x).
	                                                                                                                \]
	                                                                                                                It is immediate from the definition that $A\leqslant h_n\leqslant B$ and  $\{h_n\}_{n\in \mathbb{N}}$ is an equicontinuous sequence. By the Arzelà-Ascoli theorem we can assume that there exists an $h\in  C^{\omega}(X, \mathbb{R})$ such that $h_n$ converges uniformly  to $h$ on the  compact subsets of $X$. Hence for every $x\in X$,  
	                                                                                                                \begin{multline}
	                                                                                                                |\mathscr{L}_{\mathsf{M},f} h(x)- \rho h(x)|= \lim_{n\to \infty} | \mathscr{L}_{\mathsf{M},f} h_n(x)-\rho h_n(x)|
	                                                                                                                \\
	                                                                                                                \leqslant \dfrac{\varrho}{n}|1-\varrho^{-n}\mathscr{L}_{\mathsf{M},f} \mathbf{1}(x)|\leqslant \dfrac{\varrho}{n}(1+B)=0.
	                                                                                                                \end{multline}
	                                                                                                                Therefore, $\mathscr{L}_{\mathsf{M},f} h= \varrho h$, and obviously $h\geqslant A>0$.
	                                                                                                                
	                                                                                                                As a consequence we can state the following adaptation of Theorem \ref{Kloeckner-main-result}.

	                                                                                                                \begin{theorem}\label{Kloeckner-main-result-adapted}
	                                                                                                                	The statement of  Theorems \ref{Kloeckner-main-result} and \ref{uncountable-version} still valid if instead a compact metric space we consider $X$  being a standard Borel space equipped with a bounded metric.
	                                                                                                                \end{theorem}
	                                                                                                         \begin{remark}
Theorem \ref{Kloeckner-main-result-adapted} is not the first Ruelle theorem to address non-compact and uncountable state spaces. In \cite{CSS19}, the authors also presented a Ruelle theorem for Hölder observables, using Optimal Transport techniques. The main distinction between Theorem \ref{Kloeckner-main-result-adapted} and the one described in \cite{CSS19} lies in the consideration of decay functions and moduli of continuity, which is in turn a non-trivial improvement.
   \end{remark}       
	                                                                                                           \begin{example}
Let $X=[0,\infty)^\mathbb{N}$	  and                                                                                                       	consider the Dyson type potential $f:X\to \mathbb{R}$ given by
 \begin{equation}\label{Dyson-non-comp-0}
 f(x)=\dfrac{x_1}{1+x_1}\sum_{n=1}^{\infty}J(n) \dfrac{x_n}{1+x_n}.
 \end{equation}
where  $J(n)$  is chosen in order to ensure convergence of the above series.

We claim that $f\in C^{\varepsilon \log}(X, \mathbb{R})$. Indeed, it is sufficient to prove that the map given by  
$\pi(x)=x_i/(1+x_i)$  belongs to $C^{\varepsilon \log}(X, \mathbb{R})$, the result will follows from the fact that $C^{\varepsilon \log}(X, \mathbb{R})$ is a Banach algebra.

Firstly, 
\begin{align}\label{Dyson-non-comp}
|\pi(x)-\pi(y)|=\left| \dfrac{x_i}{1+x_i}- \dfrac{y_i}{1+y_i}\right|=
\dfrac{|x_i-y_i|}{(1+x_i)(1+y_i)}. 
\end{align}
Now observe that for each $i\in \mathbb{N}$ we have that,
\begin{equation}\label{Dyson-non-comp-2}
d(x,y)\geqslant \dfrac{1}{2^i}\min\{|x_i-y_i|,1\}
\end{equation}
The case $x_i=y_i$ is trivial, suppose $x_i\neq y_i$. Dividing (\ref{Dyson-non-comp}) by (\ref{Dyson-non-comp-2}) we get that
\begin{align*}
\dfrac{|\pi(y)-\pi(x)|}{d(x,y)} \leqslant 2^i\cdot \dfrac{|x_i-y_i|}{(1+x_i)(1+y_i)\min\{|x_i-y_i|,1\}}.
\end{align*}
By observing  that $i$ is fixed,  and considering the cases $|x_i-y_i|<1$ and $|x_i-y_i|\geqslant 1$ one can easily see that  the fraction on the  right-hand side above  is bounded for any $x_i$ and  $y_i$ belonging in $ [0, \infty).$ Remembering that  for  $t\in [0,1]$ we have that, $t\leqslant D \omega_{\varepsilon}(t)$, it follows that there must exists a constant $C$ such that 
\[
|\pi(x)-\pi(y)|\leqslant C \omega_{\varepsilon}(d(x,y)),
\]
thus concluding that   $f\in C^{\varepsilon \log}(X, \mathbb{R})$.

Now suppose $J(n)=1/n^{2+\varepsilon}$. 
 By using Lemma \ref{Lemma-flat-1} we get, for $\varepsilon>2$  that $f$ is flat with respect to $\omega_{\varepsilon-1}$. 
Now consider the transition probability kernel $\mathsf{M}=(\mu \times \delta_x)_{x\in X}$ and $\mathsf{P}=(\Pi_{x,y}^n)_{x,y,n}$ the natural coupling  defined by
\begin{equation}
\Pi^n_{x,y}=\int_{E^n}\delta_{(\bar{x}^{w}_n , \bar{y}^{w}_n)}d\mu^n(w).
\end{equation}
Theorem \ref{decay-rate} shows that $\mathsf{P}_f$ has polynomial decay rate of degree $\varepsilon$ with respect to  $\omega_{\varepsilon}(r)=\log(r_0/r)^{-\varepsilon}$.

Then  we can apply Theorem \ref{Kloeckner-main-result-adapted}  and obtain that $\mathscr{L}_{\mathsf{M}, f}$ has polynomial decay of correlations. By following the same approach adopted in Section \ref{O(N)-model} we have that, for each $\varphi \in C^{\varepsilon \log}(X, \mathbb{R})$ satisfying  $\int\varphi\, d\mu_f=0$, the potential $v=-\sum_{n=0}^\infty \mathscr{L}_{\mathsf{M}, \bar{f}}\varphi$ is a solution, thus implying the validity of a FCLT for this model.
\end{example}

\section*{Acknowledgments}
 The authors  thanks to Leandro Ciolleti and Ali Messaoudi for fruitful discussions and suggestions.

\bibliographystyle{alpha}
\bibliography{referencias}

\newcommand{\etalchar}[1]{$^{#1}$}
\begin{thebibliography}{BEvELN18}

\bibitem[Ans21]{MR4200794}
Mohammad Ansari.
\newblock Supermixing and hypermixing operators.
\newblock {\em J. Math. Anal. Appl.}, 498(1):Paper No. 124952, 13, 2021.

\bibitem[BCL{\etalchar{+}}11]{MR2864625}
A.~T. Baraviera, L.~Cioletti, A.~O. Lopes, J.~Mohr, and R.~R. Souza.
\newblock On the general one-dimensional {$XY$} model: positive and zero
  temperature, selection and non-selection.
\newblock {\em Rev. Math. Phys.}, 23(10):1063--1113, 2011.

\bibitem[BEvELN18]{Rodrigo}
Rodrigo Bissacot, Eric~O. Endo, Aernout C.~D. van Enter, and Arnaud Le~Ny.
\newblock Entropic repulsion and lack of the {$g$}-measure property for {D}yson
  models.
\newblock {\em Comm. Math. Phys.}, 363(3):767--788, 2018.

\bibitem[Bil99]{MR1700749}
P.~Billingsley.
\newblock {\em Convergence of probability measures}.
\newblock Wiley Series in Probability and Statistics: Probability and
  Statistics. John Wiley \& Sons, Inc., New York, second edition, 1999.
\newblock A Wiley-Interscience Publication.

\bibitem[BL88]{BhLe}
Rabi~N. Bhattacharya and Oesook Lee.
\newblock Asymptotics of a class of {M}arkov processes which are not in general
  irreducible.
\newblock {\em Ann. Probab.}, 16(3):1333--1347, 1988.

\bibitem[BM07]{Bhbook}
Rabi Bhattacharya and Mukul Majumdar.
\newblock {\em Random dynamical systems}.
\newblock Cambridge University Press, Cambridge, 2007.
\newblock Theory and applications.

\bibitem[BNNT22]{Barrientos2022-cr}
Pablo~G Barrientos, Fumihiko Nakamura, Yushi Nakano, and Hisayoshi Toyokawa.
\newblock Finitude of physical measures for random maps.
\newblock September 2022.

\bibitem[Bre60]{Breiman}
Leo Breiman.
\newblock The strong law of large numbers for a class of {M}arkov chains.
\newblock {\em Ann. Math. Statist.}, 31:801--803, 1960.

\bibitem[BT07]{Baladi}
Viviane Baladi and Masato Tsujii.
\newblock Anisotropic {H}\"{o}lder and {S}obolev spaces for hyperbolic
  diffeomorphisms.
\newblock {\em Ann. Inst. Fourier (Grenoble)}, 57(1):127--154, 2007.

\bibitem[CLS20]{MR4112024}
Leandro Cioletti, Artur~O. Lopes, and Manuel Stadlbauer.
\newblock Ruelle operator for continuous potentials and {DLR}-{G}ibbs measures.
\newblock {\em Discrete Contin. Dyn. Syst.}, 40(8):4625--4652, 2020.

\bibitem[CMRS21]{MR4246976}
L.~Cioletti, L.~Melo, R.~Ruviaro, and E.~A. Silva.
\newblock On the dimension of the space of harmonic functions on transitive
  shift spaces.
\newblock {\em Adv. Math.}, 385:Paper No. 107758, 24, 2021.

\bibitem[CS16]{MR3538412}
L.~Cioletti and E.~A. Silva.
\newblock Spectral properties of the {R}uelle operator on the {W}alters class
  over compact spaces.
\newblock {\em Nonlinearity}, 29(8):2253--2278, 2016.

\bibitem[CSS19]{CSS19}
L.~Cioletti, E.~Silva, and M.~Stadlbauer.
\newblock Thermodynamic formalism for topological {M}arkov chains on standard
  {B}orel spaces.
\newblock {\em Discrete \& Continuous Dynamical Systems - A},
  39(11):6277--6298, 2019.

\bibitem[DU91]{MR1014246}
M.~Denker and M.~Urba{\'n}ski.
\newblock On the existence of conformal measures.
\newblock {\em Trans. Amer. Math. Soc.}, 328(2):563--587, 1991.

\bibitem[Dug66]{Dugundji}
James Dugundji.
\newblock {\em Topology}.
\newblock Allyn and Bacon, 1966.

\bibitem[FB09]{LinDy}
Etienne~MATHERON Frédéric~Bayart.
\newblock {\em Dynamics of Linear Operators}.
\newblock Cambridge University Press, Cambridge, 2009.

\bibitem[Fog69]{foguel}
S.~R. Foguel.
\newblock {\em The ergodic theory of {M}arkov processes}.
\newblock Van Nostrand Mathematical Studies, No. 21. Van Nostrand Reinhold Co.,
  New York-Toronto, Ont.-London, 1969.

\bibitem[Fur63]{MR163345}
Harry Furstenberg.
\newblock Noncommuting random products.
\newblock {\em Trans. Amer. Math. Soc.}, 108:377--428, 1963.

\bibitem[Geo11]{MR2807681}
H-O. Georgii.
\newblock {\em Gibbs measures and phase transitions}, volume~9 of {\em de
  Gruyter Studies in Mathematics}.
\newblock Walter de Gruyter \& Co., Berlin, second edition, 2011.

\bibitem[GKLM18]{CalcuFT}
Paolo Giulietti, Beno\^{i}t Kloeckner, Artur~O. Lopes, and Diego Marcon.
\newblock The calculus of thermodynamical formalism.
\newblock {\em J. Eur. Math. Soc. (JEMS)}, 20(10):2357--2412, 2018.

\bibitem[GL78]{Gordin}
M.~I. Gordin and B.~A. Lif\v{s}ic.
\newblock Central limit theorem for stationary {M}arkov processes.
\newblock {\em Dokl. Akad. Nauk SSSR}, 239(4):766--767, 1978.

\bibitem[Gou15]{MR3309098}
S\'{e}bastien Gou\"{e}zel.
\newblock Limit theorems in dynamical systems using the spectral method.
\newblock In {\em Hyperbolic dynamics, fluctuations and large deviations},
  volume~89 of {\em Proc. Sympos. Pure Math.}, pages 161--193. Amer. Math.
  Soc., Providence, RI, 2015.

\bibitem[HH01a]{HennionHerve:2001}
H.~Hennion and L.~Herv{\'e}.
\newblock {\em Limit theorems for {M}arkov chains and stochastic properties of
  dynamical systems by quasi-compactness}, volume 1766 of {\em Lecture Notes in
  Mathematics}.
\newblock Springer-Verlag, Berlin, 2001.

\bibitem[HH01b]{MR1862393}
H.~Hennion and L.~Herv\'{e}.
\newblock {\em Limit theorems for {M}arkov chains and stochastic properties of
  dynamical systems by quasi-compactness}, volume 1766 of {\em Lecture Notes in
  Mathematics}.
\newblock Springer-Verlag, Berlin, 2001.

\bibitem[HM08]{MR2478676}
M.~Hairer and J.~C. Mattingly.
\newblock Spectral gaps in {W}asserstein distances and the 2{D} stochastic
  {N}avier-{S}tokes equations.
\newblock {\em Ann. Probab.}, 36(6):2050--2091, 2008.

\bibitem[Hop54]{hopf}
E.~Hopf.
\newblock The general temporally discrete {M}arkov process.
\newblock {\em J. Rational Mech. Anal.}, 3:13--45, 1954.

\bibitem[Klo19]{Kloeffec}
Beno\^{i}t~R. Kloeckner.
\newblock Effective perturbation theory for simple isolated eigenvalues of
  linear operators.
\newblock {\em J. Operator Theory}, 81(1):175--194, 2019.

\bibitem[Klo20]{kloeckner_2020}
B.~Kloeckner.
\newblock An optimal transportation approach to the decay of correlations for
  non-uniformly expanding maps.
\newblock {\em Ergodic Theory and Dynamical Systems}, 40(3):714--750, 2020.

\bibitem[KLO23]{MR4637163}
BENO\^{I}T~R. KLOECKNER.
\newblock An optimal transportation approach to the decay of correlations for
  non-uniformly expanding maps -- {CORRIGENDUM}.
\newblock {\em Ergodic Theory Dynam. Systems}, 43(10):3538--3544, 2023.

\bibitem[LMMS15]{MR3377291}
A.~O. Lopes, J.~K. Mengue, J.~Mohr, and R.~R. Souza.
\newblock Entropy and variational principle for one-dimensional lattice systems
  with a general {\it a priori} probability: positive and zero temperature.
\newblock {\em Ergodic Theory Dynam. Systems}, 35(6):1925--1961, 2015.

\bibitem[LMSV21]{MR4344719}
Artur~O. Lopes, Ali Messaoudi, Manuel Stadlbauer, and Victor Vargas.
\newblock Invariant probabilities for discrete time linear dynamics via
  thermodynamic formalism.
\newblock {\em Nonlinearity}, 34(12):8359--8391, 2021.

\bibitem[LY01]{MR1881258}
Ka-Sing Lau and Yuan-Ling Ye.
\newblock Ruelle operator with nonexpansive {IFS}.
\newblock {\em Studia Math.}, 148(2):143--169, 2001.

\bibitem[MT09]{MR2509253}
Sean Meyn and Richard~L. Tweedie.
\newblock {\em Markov chains and stochastic stability}.
\newblock Cambridge University Press, Cambridge, second edition, 2009.
\newblock With a prologue by Peter W. Glynn.

\bibitem[MU01]{MR1853808}
R.~D. Mauldin and M.~Urba{\'n}ski.
\newblock Gibbs states on the symbolic space over an infinite alphabet.
\newblock {\em Israel J. Math.}, 125:93--130, 2001.

\bibitem[MW00]{Wood}
Michael Maxwell and Michael Woodroofe.
\newblock Central limit theorems for additive functionals of {M}arkov chains.
\newblock {\em Ann. Probab.}, 28(2):713--724, 2000.

\bibitem[Nag57]{Nagaev}
S.~V. Nagaev.
\newblock Some limit theorems for stationary {M}arkov chains.
\newblock {\em Teor. Veroyatnost. i Primenen.}, 2:389--416, 1957.

\bibitem[Rue68]{MR0234697}
D.~Ruelle.
\newblock Statistical mechanics of a one-dimensional lattice gas.
\newblock {\em Comm. Math. Phys.}, 9:267--278, 1968.

\bibitem[Sar01a]{MR1818392}
O.~Sarig.
\newblock Thermodynamic formalism for null recurrent potentials.
\newblock {\em Israel J. Math.}, 121:285--311, 2001.

\bibitem[Sar01b]{Sarig:2001a}
Omri~M. Sarig.
\newblock {Thermodynamic formalism for null recurrent potentials.}
\newblock {\em Isr. J. Math.}, 121:285--311, 2001.

\bibitem[Sar03]{MR1955261}
O.~Sarig.
\newblock Existence of {G}ibbs measures for countable {M}arkov shifts.
\newblock {\em Proc. Amer. Math. Soc.}, 131(6):1751--1758 (electronic), 2003.

\bibitem[Sar13]{Sarig:2013}
Omri~M. Sarig.
\newblock Symbolic dynamics for surface diffeomorphisms with positive entropy.
\newblock {\em J. Amer. Math. Soc.}, 26:341--426, 2013.

\bibitem[Vil09]{MR2459454}
C.~Villani.
\newblock {\em Optimal transport}, volume 338 of {\em Grundlehren der
  Mathematischen Wissenschaften [Fundamental Principles of Mathematical
  Sciences]}.
\newblock Springer-Verlag, Berlin, 2009.
\newblock Old and new.

\end{thebibliography}
%
\end{document}